\documentclass[a4paper,12pt]{article} 
\usepackage[T1]{fontenc}
\usepackage[latin2]{inputenc}
\usepackage[a4paper,left=2cm,right=2cm,top=2cm,bottom=2cm]{geometry}
\usepackage {times,authblk,url} 
\usepackage {amsmath}
\usepackage {amsfonts}
\usepackage {amssymb}
\usepackage {bm}
\usepackage{graphicx}
\usepackage {color}

\newtheorem{theorem}{Theorem}[section]

\newtheorem{lemma}[theorem]{Lemma}

\def\Box{\raisebox{3pt}{\framebox{\hbox to 3pt{\vbox to 3pt{}}}}}
\newenvironment{proof}{\medskip \noindent{\sc Proof:}}{\quad$\Box$\par\medskip} 
\newenvironment{proofof}[1]{\medskip \noindent{\sc Proof of #1:}}{\quad$\Box$\par\medskip} 
\newenvironment{definition}[1][Definition]{\begin{trivlist}
\item[\hskip \labelsep {\bfseries #1}\stepcounter{theorem}]}{\end{trivlist}}
\newtheorem{corollary}[theorem]{Corollary}
\newenvironment{statement2}[1][Claim]{\begin{trivlist}
\item[\hskip \labelsep {\bfseries #1}\stepcounter{theorem}]}{\end{trivlist}}

\newtheorem{prop}[theorem]{Proposition}
\newcommand{\kisnyil}{\text{$\to$}}
\DeclareMathOperator{\opt}{opt}

\title{The Optimal Rubbling Number of Ladders, Prisms and M\"obius-ladders}
\author[1]{Gyula Y. Katona\thanks{kiskat@cs.bme.hu}}
\author[2]{L\'aszl\'o F. Papp\thanks{lazsa88@gmail.com}}
\affil[1]{Department of Computer Science and
Information Theory,
  Budapest University of Technology and Economics /
	MTA-ELTE Numerical Analysis and Large Networks   Research Group
}
\affil[2]{Department of Computer Science and
Information Theory,
  Budapest University of Technology and Economics}

\begin{document}
\maketitle

\begin{abstract}
  A pebbling move on a graph removes two pebbles at a vertex and adds
  one pebble at an adjacent vertex. Rubbling is a version of pebbling
  where an additional move is allowed. In this new move, one pebble
  each is removed at vertices $v$ and $w$ adjacent to a vertex $u$,
  and an extra pebble is added at vertex $u$. A vertex is reachable
  from a pebble distribution if it is possible to move a pebble to
  that vertex using rubbling moves.  The optimal rubbling number is
  the smallest number $m$ needed to guarantee a pebble distribution of
  $m$ pebbles from which any vertex is reachable. We determine the
  optimal rubbling number of ladders ($P_n\square P_2$), prisms
  ($C_n\square P_2$) and M\"oblus-ladders.
\end{abstract}


\section{Introduction}

Graph pebbling has its origin in number theory. It is a model for the
transportation of resources. Starting with a pebble distribution on
the vertices of a simple connected graph, a \emph{pebbling move}
removes two pebbles from a vertex and adds one pebble at an adjacent
vertex. We can think of the pebbles as fuel containers. Then the loss
of the pebble during a move is the cost of transportation. A vertex is
called \emph{reachable} if a pebble can be moved to that vertex using
pebbling moves. There are several questions we can ask about
pebbling. One of them is: How can we place the smallest number of
pebbles such that every vertex is reachable (\emph{optimal pebbling
  number})? For a comprehensive list of references for the extensive
literature see the survey papers
\cite{Hurlbert_survey1,Hurlbert_survey2, Hurlbert_survey3}.

\emph{Graph rubbling} is an extension of graph pebbling. In this
version, we also allow a move that removes a pebble each from the
vertices $v$ and $w$ that are adjacent to a vertex $u$, and adds a
pebble at vertex $u$. The basic theory of rubbling and optimal
rubbling is developed in \cite{BelSie}. The rubbling number of
complete $m$-ary trees are studied in \cite{Danz}, while the rubbling
number of caterpillars are determined in \cite{Papp}. In \cite{KaSi}
the authors give upper and lower bounds for the rubbling number of
diameter 2 graphs.

In the present paper we determine the optimal rubbling number of
ladders ($P_n\square P_2$), prisms ($C_n\square P_2$) and
M\"oblus-ladders.

\section{Definitions}

Throughout the paper, let $G$ be a simple connected graph. We use the
notation $V(G)$ for the vertex set and $E(G)$ for the edge set.  A
\emph{pebble function} on a graph $G$ is a function
$p:V(G)\to\mathbb{Z}$ where $p(v)$ is the number of pebbles placed at
$v$. A \emph{pebble distribution} is a nonnegative pebble
function. The \emph{size} of a pebble distribution $p$ is the total
number of pebbles $\sum_{v\in V(G)}p(v)$. If $H$ is a subgraph of $G$, then $p(H)=\sum_{v\in V(H)}p(v)$.  We say that a vertex $v$ is
\emph{occupied} if $p(v)>1$, else it is \emph{unoccupied}. 

Consider a pebble function $p$ on the graph $G$. If $\{v,u\}\in E(G)$
then the \emph{pebbling move} $(v,v\kisnyil u)$ removes two pebbles at
vertex $v$, and adds one pebble at vertex $u$ to create a new pebble
function $p'$, so $p'(v)=p(v)-2$ and $p'(u)=p(u)+1$. 
If $\{w,u\}\in E(G)$ and $v\not=w$, then the \emph{strict rubbling
move} $(v,w\kisnyil u)$ removes one pebble each at vertices $v$ and $w$,
and adds one pebble at vertex $u$ to create a new pebble function
$p'$, so $p'(v)=p(v)-1$, $p'(w)=p(w)-1$   and $p'(u)=p(u)+1$.

A \emph{rubbling move is} either a pebbling move or a strict rubbling
move. A \emph{rubbling sequence} is a finite sequence $T=(t_{1},\ldots,t_{k})$
of rubbling moves. The pebble function obtained from the pebble function
$p$ after applying the moves in $T$ is denoted by $p_{T}$. The concatenation
of the rubbling sequences $R=(r_{1},\ldots,r_{k})$ and $S=(s_{1},\ldots,s_{l})$
is denoted by $RS=(r_{1},\ldots,r_{k},s_{1},\ldots,s_{l})$.

A rubbling sequence $T$ is \emph{executable} from the pebble distribution
$p$ if $p_{(t_{1},\ldots,t_{i})}$ is nonnegative for all $i$. A
vertex $v$ of $G$ is \emph{reachable} from the pebble distribution
$p$ if there is an executable rubbling sequence $T$ such that $p_{T}(v)\ge1$.
$p$ is a solvable distribution when each vertex is
reachable. Correspondingly, $v$ is \textit{$k$-reachable} under $p$ if
there is an executable $T$, that $p_T(v)\geq k$, and $p$ is
\textit{k-solvable} when every vertex is $k$-reachable. An $H$ subgraph
is \textit{$k$-reachable} if there is an executable rubbling sequence
$T$ such that $p_T(H)=\sum_{v\in V(H)}p_T(v)\geq k$. We say that vertices $u$ and $v$ are
\textit{independently reachable} if there is an executable rubbling
sequence $T$ such that $p_T(u)\geq 1$ and $p_T(v)\geq 1$.

The \emph{optimal rubbling number} $\varrho_{\opt}(G)$ of a graph
$G$ is the size of a distribution with the least number of pebbles
from which every vertex is reachable. A solvable pebbling distribution is \emph{optimal} if its size equals to the optimal rubbling number.

Let $G$ and $H$ be simple graphs. Then the \textit{Cartesian
    product} of graphs $G$ and $H$ is the graph whose vertex set is
  $V(G)\times V(H)$ and $(g,h)$ is adjacent to $(g',h')$ if and only
  if $g=g'$ and $(h,h')\in E(H)$ or if $h=h'$ and $(g,g')\in E(G)$. This graph
  is denoted by $G\square H$.

  $P_n$ and $C_n$ denotes the path and the cycle containing $n$
  distinct vertices, respectively. We call $P_n\square P_2$ a \emph{ladder}
  and $C_n\square P_2$ a \emph{prism}. It is clear that the prism can be
  obtained from the ladder by joining the 4 endvertices by two edges
  to form two vertex disjoint $C_n$ subgraphs. If the four endvertices
  are joined by two new edges in a switched way to get a $C_{2n}$
  subgraph, then a \emph{M\"obius-ladder} is obtained.

  We imagine the $P_n\square P_2$ ladder laid horizontally, so there
  is an upper $P_n$ path, and a lower $P_n$ path, which are connected
  by ``parallel'' edges, called \emph{rungs} of the ladder. Vertices
  on the upper path will be usually denoted by $v_i$, while vertices of
  the lower path by $w_i$. Also, if $A$ is a rung (a vertical edge of
  the graph), then $\overline{A}$ denotes the upper, and
  $\underline{A}$ the lower endvertex of this rung. This arrangement
  also defines a natural left and right direction on the horizontal
  paths, and between the rungs.

\section{Optimal rubbling number of the ladder}

In this section we give a formula for the optimal rubbling number of ladders: 

\begin{theorem}
  Let $n=3k+r$ such that $0\leq r<3$ and $n,r\in \mathbb{N}$, so
  $k=\left\lfloor\frac{n}{3}\right\rfloor$. Then
$$\varrho_{\opt}(P_n \square P_2)=\left \{ \begin{array}{ll}
1+2k & \text{if}\ r=0,\\
2+2k& \text{if}\ r=1,\\
2+2k & \text{if}\ r=2.
 \end{array}\right.$$
\label{letratetel}
\end{theorem}

In the rest of the section we are going to prove the above
theorem. The proof is fairly long and complex, so it will be divided
to several lemmas. First, we prove that the function given in the
theorem is an upper bound, by giving solvable distributions.

\begin{lemma}
  Let $n=3k+r$ such that $0\leq r<3$ and $n,r\in \mathbb{N}$, so
  $k=\left\lfloor\frac{n}{3}\right\rfloor$.
$$\varrho_{\opt}(P_n \square P_2)\leq \left \{ \begin{array}{ll}
1+2k & if\ r=0,\\
2+2k& if\ r=1,\\
2+2k & if\ r=2.
 \end{array}\right.$$
\label{letratetelfelso}
\end{lemma}

\begin{proof}
A solvable distribution with adequate size is shown in Fig. \ref{felso} for each case.
\end{proof}

\begin{figure}[htb]
\centering
\includegraphics[scale=0.6]{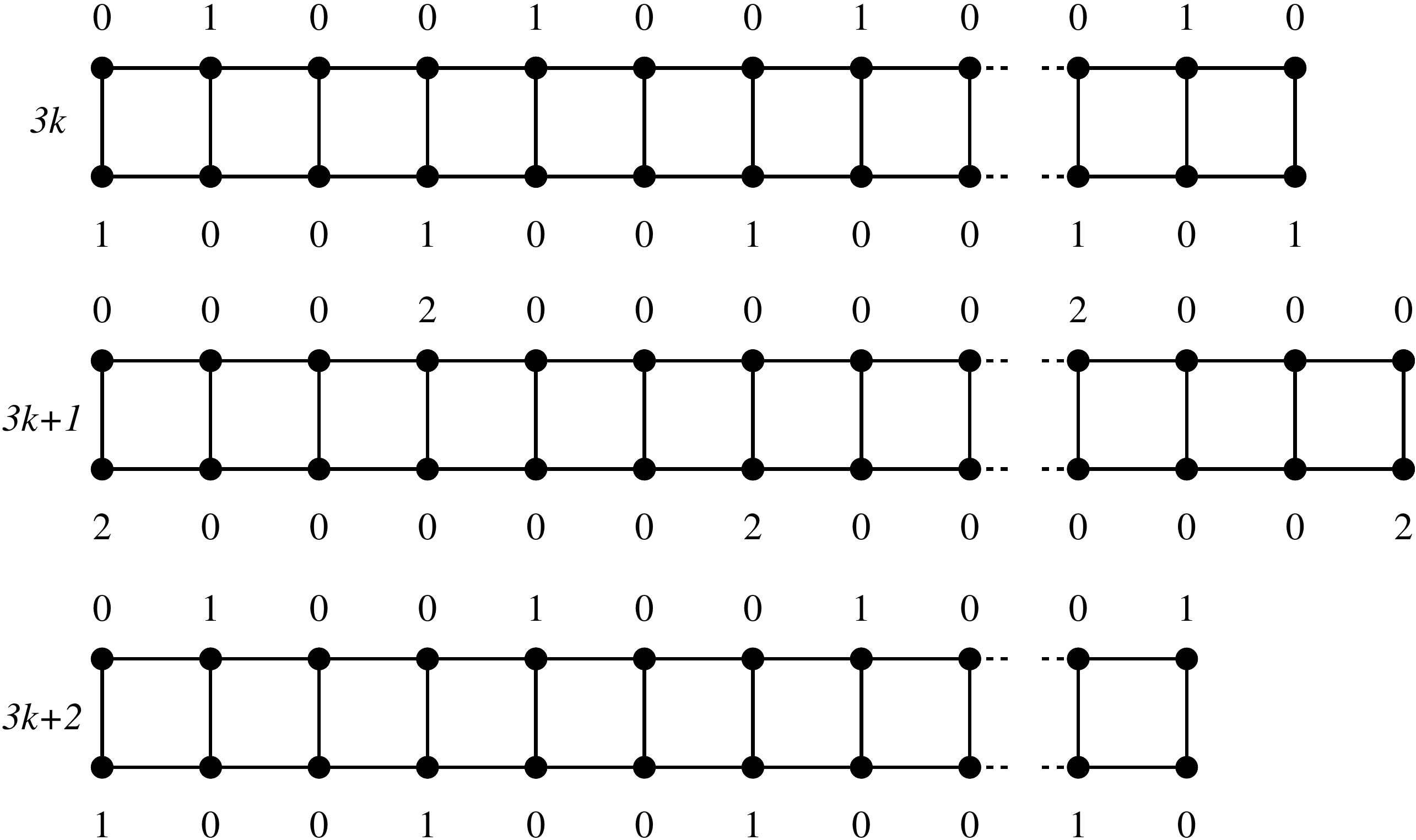}
\caption{Optimal distributions.}
\label{felso}
\end{figure}

\label{bizkezdet}
Now we need to prove that the function is a lower bound as well. This
part is unfortunately much harder. Before we start the rigorous proof,
a summary of the proof is given. Then the necessary definitions and
proofs of several Lemmas will follow.\medskip

\noindent\textsc{Summary of the proof:} 
We prove by induction on $n$. First we deal with the base cases in
Lemma \ref{lem:optimalis}.  For the induction step, consider an 
optimal distribution $p$ on $P_n \square P_2$. Choose an appropriate
$R=P_3\square P_2$ subgraph which contains maximum number of pebbles,
delete the vertices of $R$ and reconnect the remaining two parts to
obtain $G^R=P_{n-3} \square P_2$, called the \emph{reduced graph}, see
Fig. \ref{terv}.
\begin{figure}[htb]
\centering
\scalebox{0.6}{\input{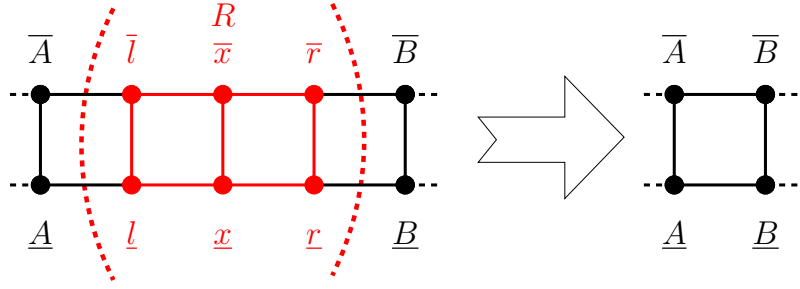}}
\caption{Deleting an $R=P_3\square P_2$ subgraph lying between the two dashed brackets.}
\label{terv}
\end{figure} 

Now construct a solvable $p^R$ distribution for the new $P_{n-3}\square
P_2$ graph in the following way: $p$ induces a distribution on the
vertices which we have not deleted. In most of the cases we simply
place $p(v)$ pebbles to all $v\in V(G)\backslash V(R)$, (i.e. do not
change the original distribution), in some other cases we apply a
simple operation on the original distribution.
 Finally, distribute and place $p(R)-2$ pebbles
at vertices $\overline{A}$, $\underline{A}$, $\overline{B}$ and
$\underline{B}$ in an appropriate way so that the new distribution on
$P_{n-3} \square P_2$ is solvable. Our aim is to show that it is
always possible to find such a new distribution. This will be proved in
several lemmas. These will imply
$$\varrho_{\opt}(P_n\square P_2)\geq \varrho_{\opt}(P_{n-3}\square P_2)+2,$$
which implies the theorem.
The most challenging part of the proof is to show that  the reduced distribution is solvable on the reduced graph. The obvious idea to do this is to show that it is possible to change any rubbling sequence of the original graph such that it reaches the same vertices. However, there are way too many and very tricky rubbling sequences on the ladder (the path and the cycle was much easier in this sense). So first we show that for a solvable distribution and any vertex $v$ there exist a ``nice'' (so called $A$-biased) rubbling sequence that reaches $v$ (see Lemma~\ref{A-max}). Next we invent a new way to show that the reduced distribution remains solvable using the existence of these ``nice'' rubbling sequences (see  Corollary~\ref{p-kov}).

To complete the proof a detailed case analysis is needed to treat all the essentially different pebble distributions of the deleted subgraph $R$. Lemma~\ref{lem:legalabbnegy} show why this method is useful. The case when $R$ contains at least 4 pebbles is treated in Subsection~\ref{minnegy}, the cases when $R$ it contains three or two pebbles are handled in Subsections~\ref{maxharom} and \ref{maxketto}.
Finally, the proof is put together on page~\pageref{fobiz}.
\medskip

Now we start the detailed proof with some lemmas.

\begin{lemma}\label{lem:optimalis}
$$\varrho_{\opt}(P_2)=2$$
$$\varrho_{\opt}(P_2\square P_2)=2$$
$$\varrho_{\opt}(P_3\square P_2)=3$$
\end{lemma}

\begin{proof}
  The optimal distributions are shown in Fig. \ref{optimalis}. It is
  an easy exercise to check that these distributions are optimal.
\end{proof}
\begin{figure}[htb]
\center
\includegraphics[scale=0.6]{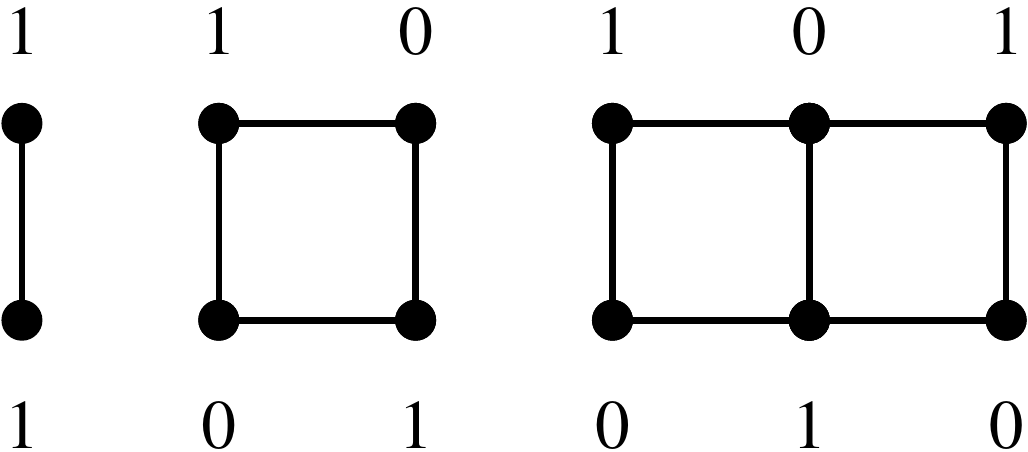}
\caption{Optimal distributions for $P_2$, $P_2\square P_2$ and $P_3\square P_2$.}\label{optimalis}
\end{figure}

Now we make some preparations to prove the lower bound. As mentioned
in the summary, we need to give a new distribution $p^R$ on the
reduced graph $G^R$. Next we give some properties that all such
distributions needs to satisfy.
\begin{definition}
  The distribution we need to construct on $G^R$ (mentioned on Page \pageref{bizkezdet} in the proof summary) is denoted by
  $p^R$, called \emph{reduced distribution}.  It needs to satisfy
  the following conditions:
\begin{itemize}
\item $p^R(v)=p(v)$ or $p^R(v)=p^R(w)$, if $v$ is not contained by rung $A$ or $B$ and $v$ and $w$ contained in the same rung.
\item $p^R(R)=p(R)$ if $R$ is a rung other than $A$ and $B$.
\item $p^R(w)\geq p(w)$, if $w$ is contained by rung $A$ or $B$.
\item $p^R(\overline{A})+p^R(\underline{A})+p^R(\overline{B})
+p^R(\underline{B})=p(\overline{A})+p(\underline{A})+p(\overline{B})
+p(\underline{B})+p(R)-2$.
\end{itemize}

\end{definition}

One of the tools that is used in the proof is the ``weight
argument''. This was introduced by Moews in \cite{Moews}, now it is
extended for our situation.
\begin{definition}
  Let $d(x,v)$ denote the distance between vertices $x$ and $v$,
  i.e.~ the length of the shortest path which connects them.  The
  \emph{weight-function} of a vertex $x$ with respect to pebble
  distribution $p$ is:
  \[w_{p}(x)=\sum_{v\in V(G)}\left(\frac{1}{2}\right)^{d(x,v)}p (v).\]
  The \emph{left weight-function}, denoted by $Lw_{p}(x)$, is a similar
  function, the difference is that the summation is taken only for
  vertices that do not lie to the right of $x$ (i.e. for vertices lying
  left and the other vertex of the rung containing $x$).  The
  \emph{right weight-function}, denoted by $Rw_{p}(x)$ is defined
  similarly.
\end{definition}

\begin{definition}
 Let $p$ be a distribution on the graph $G=P_n\square P_2$. Fix a
  vertex $v$ and delete every vertex  located to the right of it. We get a
  shorter $G'=P_{m}\square P_2$ graph, $m\leq n$, which does not contain vertices located to the right
  of $v$. Let $p'$ be a pebble distribution on $G'$ such that
  $p'(v)=p(v)$ for each vertex of $G'$. We say that $v$ is \textit{left
    $k$-reachable} in $G$ if it is $k$-reachable in  $G'$ under the
  distribution $p'$. \emph{Right $k$-reachability} is defined similarly.
  Let $L_p(v)$ (and $R_p(v)$) denote the maximum $k$ for which $v$ is
  left-$k$-reachable (right-$k$-reachable).
\end{definition}

The following two lemmas show the connection between the left (right)
weight-function and $L_p$ ($R_p$). These show us that we can
approximate $L_p$ with $Lw_p$ and $R_p$ with $Rw_p$. This will be used
on page \pageref{modified}.

\begin{lemma}\label{weight-lemma}
 $L_p(v)\leq Lw_p(v)$
and $R_p(v)\leq Rw_p(v)$
hold for any vertex $v$.
\end{lemma}

\begin{proof}
It is clear that a rubbling step cannot increase the value of the left (right) weight-function at $v$. However, if a sequence $T$ of rubbling steps moved $k$ pebbles  to $v$ from the left, then \[k\leq Lw_{p_T}(v)\leq Lw_p(v)\] holds, proving the first claim. The second claim can be proved similarly. 
\end{proof}

In fact, a stronger statement can be proved.

\begin{lemma}
\label{egeszresz}
$L_{p}(v)=\left\lfloor Lw_{p}(v)\right\rfloor$ and
$R_{p}(v)=\left\lfloor Rw_{p}(v)\right\rfloor$ hold for any vertex $v$.
\end{lemma}

\begin{proof}
We only verify the first claim; the second claim can be handled similarly.
In the following, we only consider pebbles and vertices that are not on the right hand side of $v$.  

There are at most two vertices $w$ and $w'$ in the graph whose
distance from $v$ is $d$ and they are not located to the left of
$v$ (if $d$ is too large, then one or zero). Vertices $w$ and $w'$ have at most two common neighbours. Let $u$
be the one which is closest to $v$. Move as many pebbles as possible
from $w$ and $w'$ to this neighbour $u$ by rubbling moves. Use the
same moves for every $d$ in decreasing order, to obtain a distribution
$p'$. Let us call this \emph{greedy rubbling}. As a result, in $p'$,
the two vertices at distance $d>0$ from $v$ contains at most one
pebble together. It is easy to see that
$Lw_{p'}(v)=Lw_{p}(v)$. Therefore
\begin{align*}
  Lw_{p}(v)&=Lw_{p'}(v)=\sum_{x\in V(G)}\frac{1}{2^{d(x,v)}}p'(x)=\\
  &=p'(v)+\sum_{x\neq v}\frac{1}{2^{d(x,v)}}p'(x)\leq
  p'(v)+\sum_{d=1}^k \frac{1}{2^d}<L_{p'}(v)+1=L_{p}(v)+1.
\end{align*}
By Lemma \ref{weight-lemma} and the fact that $L_{p}(v)$ is an integer, the claim is proved.
\end{proof}

\begin{definition}
  Let $p$ be a distribution on $P_n\square P_2$ and let $A$ be a
  rung. An executable rubbling sequence $S$ is called
  \textit{$A$-biased} if each rubbling move that takes a pebble from
  $A$ to another rung only use pebbles from the same vertex of
  $A$. So, if $S$ is $A$-biased and $(\underline{A},v\kisnyil w)\in S$
  where $w \notin V(A)$, then $(\overline{A},v'\kisnyil w') \notin S$
  except in the case when $w'=\underline{A}$. This also holds if we swap
  $\overline{A}$ and $\underline{A}$.
\end{definition}

We invent this notion for the following reason.
Assume that a vertex $v$ located to the left from $A$ is reachable by
an $A$-biased sequence $S$ under distribution $p$. Furthermore, assume
that all moves of $S$ taking a pebble from $A$ to another rung use
only pebbles from $\underline{A}$. Let $q$ be a modification of $p$
such that $q(u)=p(u)$ where $u\neq \underline{A}$ and
$q(\underline{A})=R_p(\underline{A})$. We can make a new sequence
which acts only on $\underline{A}$ and vertices to the left from $A$
and still reaches $v$ under $q$. Finally, if we modify the graph or
the distribution on the right hand side of $A$, then $v$ remains
reachable if $\underline{A}$ remains right
$R_p(\underline{A})$-reachable.

\par
It will be done in several steps. In Lemma \ref{kiveteleslemma} we
show that all but one vertex lying to the left of $A$ is reachable by
an $A$-biased sequence. In the next lemma we show that we can modify
$p$ to $q$ such that their size is the same, each vertex located to
the left of $A$ is reachable by an $A$-biased sequence and every vertex not located
 to the left of $A$ is remains
reachable. 
The first corollary shows that we can do it for two different rungs
simultaneously if we consider opposite orientations of the graph. In
Corollary \ref{p-kov} the notion of $A$-biased sequence pays off;
we get four inequalities which are sufficient conditions for the
reachability of any vertex not contained in rung $A$ or $B$.  This
makes the rest of the proof substantially easier.

\begin{lemma} \ 
\begin{itemize}
\item[i)] When $S$ is an $A$-biased sequence, $T$ is a sequence which does not contain a move acting on rung $A$, and $ST$ is executable, then $ST$ is $A$-biased.  
\label{A-maxtriv}
\item[ii)]
The greedy rubbling sequence is A-biased.
\end{itemize}
\end{lemma}

\begin{proof} These statements are direct consequences of the definition of $A$-biased sequences.
\end{proof}

\begin{lemma}
Let $p$ be a solvable distribution of $P_n\square P_2$. Let $A$ be an arbitrary rung in the graph, and let $S_s$ denote an $A$-biased sequence that reaches vertex $s$. Such an $S_s$ exists for all but one vertex located left from $A$.
Furthermore, if there is an exception then we must see the distribution shown in Fig. \ref{exceptionfig} during the execution of the rubbling sequence reaching $s$. (We have assumed on the figure that the exception vertex is $v_i$. 
This exception vertex can be different for different rungs.)
\label{kiveteleslemma}
\end{lemma}

\begin{figure}[bht]
\centering
\scalebox{0.6}{\input{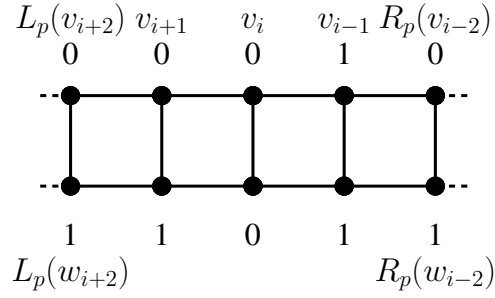}}
\caption{The only possible exception.}
\label{exceptionfig}
\end{figure}

\begin{proof}
  Assume that for each $s$ which is located to the right from $v_i$ and
  located to the left from rung $A$, there exist a suitable $A$-biased
  sequence $S_s$. Now we show that either some $S_{v_i}$ exists or
  $v_i$ is the single exception.  Let $S$ be a rubbling sequence such
  that $p_S(v_i)=1$. If $R_p(v_i)\geq 1$ then the greedy rubbling
  sequence towards $v_i$ from the right is executable and
  $A$-biased. $L_p(v_i)\geq 1$ means that $v_i$ is reachable without
  any pebble of rung $A$, hence the statement holds trivially in this
  case. Thus we have to check cases where $R_p(v_i)=0$ and
  $L_p(v_i)=0$.
\par Our assumption implies that $v_{i-1}$ is reachable with an $A$-biased sequence. 
 Let $T$ be a subsequence of $S$ such that $T$ contains only rubbling moves which act on vertices located right from $v_i$, and $T$ is maximal. 
 $R_p(v_i)=0$ implies that $R_{p_T}(v_i)=0$ which means that one of the following cases holds:
\begin{itemize}
\item \textbf{Case 1.} $p_T(v_{i-1})=1$, $p_T(w_{i-1})=1$
\item \textbf{Case 2.} $p_T(v_{i-1})=1$, $p_T(w_{i-1})=0$
\item \textbf{Case 3.} $p_T(v_{i-1})=0$, $p_T(w_{i-1})=1$
\item \textbf{Case 4.} $p_T(v_{i-1})=0$, $p_T(w_{i-1})=2$
\item \textbf{Case 5.} $p_T(v_{i-1})=0$, $p_T(w_{i-1})=3$
\end{itemize}    
\par 

In Cases 3--5  we can replace $T$ with a greedy sequence towards $w_{i-1}$, denoted by $Z$, its moves also act only on vertices located right from $v_i$, so $p_T(w_{i-1})\leq p_Z(w_{i-1})$. In Case 2 we replace $T$ with a similar greedy $Z$ that reaches $v_{i-1}$. 
 $v_i$ is reachable by the executable sequence $Z(S\backslash T)$, hence this sequence is $A$-biased by Lemma \ref{A-maxtriv}. So we completed the proof for Cases 2--5. 

\par Now let us prove Case 1.

\par If $L_p(v_{i+1})\geq 1$ then $(v_{i-1},v_{i+1}\kisnyil v_i)$ can move a pebble to $v_i$ after we apply $T$ and some moves which act only on vertices not to the right from $v_{i+1}$. Thus we do not need a pebble at $w_{i-1}$, so $T$ can be replaced again by a greedy sequence towards $v_{i-1}$.  Now we show that if $L_p(v_{i+1})=0$, then we see the distribution during the execution of $S$ shown in Fig. \ref{exceptionfig}.

$R_p(v_i)=0$ implies $R_p(v_{i-1})\leq 1 $. The reachability of $v_i$, $L_p(v_{i+1})=0$ and $R_p(v_{i-1})\leq 1 $ imply that we can move a pebble to $w_{i-1}$, and for the same reasons it can be done only by the execution of a $(w_{i+1},w_{i-1}\kisnyil w_i)$ move. Thus $Lp(w_{i+1})=1$. 
The conditions $R_p(v_i)=0$, $L_p(v_{i+1})=0$, $L_p(w_{i+1})=1$, $p(w_i)=0$  imply that the distribution shown in Fig. \ref{exceptionfig} has to be seen during the reach of $v_i$. 

\par
Finally we prove that at most one exception may exist. Assume that
$v_i$ is an exception. It is easy to see that if $v_i$ is an exception
then $w_i$ can not be. Also, no vertex located to the right from $v_i$
can be an exception. We can reach $v_{i+1}$ with the
$(w_{i},w_{i+2}\kisnyil w_{i+1}),(w_{i+1},w_{i+1}\kisnyil v_{i+1})$
sequence after we reach $w_{i}$ with an $A$-biased sequence. Any other
vertex located to the left from $v_i$ can not use a pebble at $v_i$ or
$w_i$, hence we do not need to use any pebbles of rung $A$ to reach
them.
\end{proof}

\begin{lemma}
  Let $p$ be a solvable distribution of $P_n\square P_2$, and let $A$
  be an arbitrary rung in the graph. Then there exists a solvable
  distribution $q$ satisfying the following conditions:
\begin{enumerate}
\item $|q|=|p|$
\item $q(v)=p(v)$ for all vertices $v$ not located to the left of $A$.
\item There exists a sequence $S_s$ for all vertices $s$ located on
  the left  side of $A$ which is $A$-biased and reaches $s$ from
  $q$.
\item If $T$ is an executable sequence under $p$ then there is an
  executable sequence $T'$ under $q$ such that
  $p_T(\overline{A})=q_{T'}(\overline{A})$ and
  $p_T(\underline{A})=q_{T'}(\underline{A})$.
\end{enumerate}
\label{A-max} 
\end{lemma}

\begin{proof}
  If we do not get an exception while applying Lemma
  \ref{kiveteleslemma} then $q= p$ trivially satisfies all conditions,
  so we are done. Otherwise, we have to change $p$. Assume that $v_i$
  is the exceptional vertex. Let $q$ be the following distribution:
  $q(s)=p(s)$ for all vertices not located to the left from $v_i$ and
  $q(w_j)=p(v_j)$, $q(v_j)=p(w_j)$ when $j>i$. (In other words, we
  just reflect the vertices located to the left from $v_i$ on a horizontal
  axis.) Conditions 1 and 2 trivially hold again, as well as Condition
  3 for all vertices except
  ${v_i,v_{i+1},w_i,w_{i+1}}$. $L_q(v_{i+1})\geq 1$ so $v_i$ is not an
  exception under $q$. $R_p(w_i)\geq 1$ and nothing has changed at
  vertices which are not to the left of $w_i$, so $R_q(w_i)\geq
  1$. This means that Condition 3 holds for these vertices, too.

  The last condition is trivial if $p= q$, otherwise none of the
  pebbles placed on the reflected vertices can be moved to rung $A$,
  because of the definition of the exception
  ($L_p(v_i)=L_p(w_i)=0$). So let $T'$ be the part of $T$ which acts
  only on vertices located to the right from $v_i$. The fact that $p$ and $q$
  are the same on these vertices implies that $T'$ is executable. Thus
  the vertices of rung $A$ and the other vertices located to the right from
  $A$ can be reached from $q$.
\end{proof}

Naturally, the ``right-sided'' version of the above ''left-sided''
lemma can be proved similarly. Now, since the distribution in the
left-sided version is not changed on the right side, and in the
right-sided version it is not changed on the left side, we can apply
both versions simultaneously.

\begin{corollary}
  Fix rungs $A$ and $B$ such that $B$ lies to the right of $A$. Then
  there is a $q$ which fulfills the conditions of Lemma \ref{A-max}~
  for $A$ on the left side and for $B$ on the right side.
\label{A-maxkov}
\end{corollary}
\begin{proof} We apply Lemma \ref{A-max}\ first for rung $A$ then for rung $B$ with opposite orientation of the graph. The 2. and 4. condition of the lemma guarantee that we do not ruin what we got after the first application.   
\end{proof}

\begin{corollary}
Let $p$ be a solvable pebble distribution on the graph $G=P_n\square P_2$. If a distribution $p^R$ in the graph $G^R$ satisfies
\begin{itemize}
\item $R_p(\overline{A})\leq R_{p^R}(\overline{A})$,
\item $R_p(\underline{A})\leq R_{p^R}(\underline{A})$,
\item $L_p(\overline{B})\leq L_{p^R}(\overline{B})$ and
\item $L_p(\underline{B})\leq L_{p^R}(\underline{B})$  
\end{itemize} then
all vertices located to the left of $A$ and located to the right of $B$ are reachable from $p^R$.
\label{p-kov} 
\end{corollary}

\begin{proof} By Corollary \ref{A-maxkov} we can replace $p$ with $q$
  which has the same size, and all vertices located to the left of $A$
  (to the right of $B$) are reachable with an $A$-biased ($B$-biased)
  sequence. It is easy to see that
  $R_p(\overline{A})=R_{q}(\overline{A})$,
  $R_p(\overline{A})=R_{q}(\overline{A})$,
  $L_p(\overline{B})=L_{q}(\overline{B})$ and
  $L_p(\underline{B})=L_{q}(\underline{B})$ hold.

\par   
Let $p^R$ be a distribution such that the conditions hold and
$p^R(v)=q(v)$ for all $v$ located to the left from $A$, or to the right from
$B$. Fix a $v$ to the left from $A$. It suffices to show the statement for a
$v$ that is on the left  side of
$A$.

\par An $A$-biased sequence $S$ reaches $v$ under $q$ in $G$. Assume
that $S$ does not contain a $(\underline{A},*\kisnyil w)$ move where
$w\neq \overline{A}$. Let $T$ be a subsequence of $S$ such that $T$
contains only moves that act on vertices located left from rung $A$
and on $\overline{A}$. $T$ uses only $R_p(\overline{A})$ pebbles at
$\overline{A}$ and reaches $v$ under $q_{S\backslash T}$. $S\backslash
T$ moves at most $R_q(\overline{A})=R_p(\overline{A})$ pebbles to
$\overline{A}$. Let $Z$ be an executable sequence under $p^R$, such
that $Z$ does not act on any vertex left from $A$ and moves
$R_{p^R}(\overline{A})$ pebbles on $\overline{A}$. $ZT$ is executable
under $p^R$ and reaches $v$.

\par The proof is similar in all other cases.   
\end{proof}

The combination of Lemma \ref{egeszresz} and Corollary \ref{p-kov} shows that it is enough to find a distribution $p^R$ on $G^R$ that satisfies the following inequalities:
$$\left\lfloor Rw_{p^R}(\overline{A})\right\rfloor-\left\lfloor Rw_{p}(\overline{A})\right\rfloor\geq 0$$
$$\left\lfloor Rw_{p^R}(\underline{A})\right\rfloor-\left\lfloor Rw_{p}(\underline{A})\right\rfloor\geq 0$$
$$\left\lfloor Lw_{p^R}(\overline{B})\right\rfloor-\left\lfloor Lw_{p}(\overline{B})\right\rfloor\geq 0$$
$$\left\lfloor Lw_{p^R}(\underline{B})\right\rfloor-\left\lfloor Lw_{p}(\underline{B})\right\rfloor\geq 0$$
For the sake of simplicity we call these inequalities \textit{original}. The original inequalities contain floor functions. Calculating without floor functions is much easier, hence we prefer to calculate with the following \textit{modified} inequalities:
\label{modified}
$$Rw_{p^R}(\overline{A})-Rw_{p}(\overline{A})\geq 0$$
$$Rw_{p^R}(\underline{A})-Rw_{p}(\underline{A})\geq 0$$
$$Lw_{p^R}(\overline{B})- Lw_{p}(\overline{B})\geq 0$$
$$Lw_{p^R}(\underline{B})-Lw_{p}(\underline{B})\geq 0$$
It is clear that if the modified inequalities hold then the original inequalities hold as well. On the other hand, the following lemma  shows that the modified inequalities with a weak additional property imply that $p^R$ is solvable.

\begin{lemma}
If $p(R)\geq 4$ and the modified inequalities are satisfied then $p^R$ is solvable.\label{lem:legalabbnegy}
\end{lemma}

 We use the notations of Fig. \ref{terv} in the next proof and in further sections.

\begin{proof}
By the above results we only need to check that $\overline{A}$, $\underline{A}$, $\overline{B}$, $\underline{B}$ are all reachable. By symmetry it is enough to do it for $\underline{A}$. 

Assume that $p^R(\underline{A})=0$ and $p^R(\overline{A})<2$, otherwise $\underline{A}$ is trivially reachable under $p^R$.  
 After the reduction, at least two pebbles will be placed somehow on the subgraph induced by rung A and B. Now it is easy to verify that if $L_p({w_1})\geq 1$ then $\underline{A}$ is reachable under $p^R$. 
\begin{figure}
\label{tcases}
\centering
\scalebox{0.6}{\input{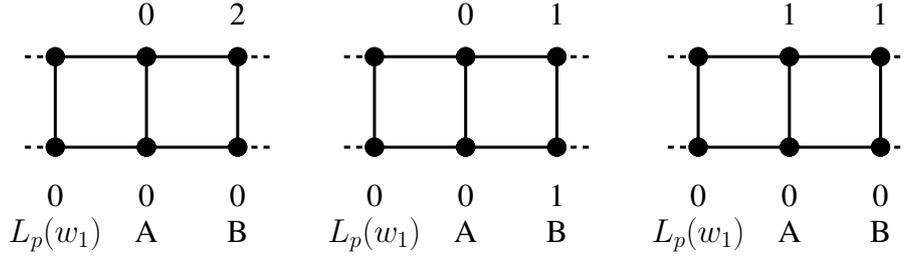}}

\caption {Three cases.}
\end{figure}

If $L_p({w_1})= 0$ then there are three remaining ways to distribute the two pebbles on $A$ and $B$, these are shown on Fig. 
5.

 $R_p(A)=0$, $L_p(w_1)=0$ and the fact that $\underline{A}$  is reachable under $p$ implies that $L_p(v_1)=1$. The reduction leaves $v_1$ reachable from the left. It is easy to see that $\underline{A}$ is reachable with the help of this pebble in the second and the third case.
In the first case, first we show that $R_{p}(\overline{l})\geq 3$ which will imply that $R_{p^R}(\overline{B})\geq 3$ and $\underline{A}$ is reachable again. 

$R_p(\overline{A})=1$, otherwise $w_1$ can not be reachable under $p$. This implies that $R_p(\overline{l})\geq 2$. The reachability of $\underline{A}$ under $p$ requires that $R_p(l)\geq 3$, but the fact that $R_p(\underline{A})=0$ and $R_p(l)\geq 3$ excludes that $R_p(\underline{l})> 0$, thus $R_p(\overline{l})\geq 3$. 

 Assume that $R_{p^R}(\overline{B})= 2$ and use the condition of $\overline{A}$ and Lemma \ref{egeszresz}. This results in the following contradiction:
$$\frac{3}{2}\leq Rw_p(\overline{A}) \leq Rw_{p^R}(\overline{A})=\frac{Rw_{p^R}(\overline{B})}{2}<\frac{R_{p^R}(\overline{B})+1}{2}=\frac{3}{2}$$

\end{proof}

\begin{corollary}
$p^R$ is solvable if one of the following statements holds:
\begin{enumerate}
\item The modified inequalities hold and $p(R)\geq 4$.
\item The original inequalities hold and the vertices of rung A and B are reachable from $p^R$.
\end{enumerate} 
\end{corollary}

\par
The elements of the graph family $P_n\square P_2$ have several
symmetries. Hence we can assume without loss of generality that
$p(\overline{l})+p(\underline{l})\geq
p(\overline{r})+p(\underline{r})$ and
$p(\overline{l})+p(\overline{x})\geq
p(\underline{l})+p(\underline{x})$.  This assumption reduces the number of cases when
we enumerate and check possibilities in the appendix
and in Fig. \ref{kivetelek} and Fig. \ref{3asok}.

\par 
In the next sections we show that a solvable reduced distribution
always exist. This requires a case analysis. In most of the cases it is
enough to use the modified inequalities, however, in a few cases the
original inequalities are needed.

\subsection{The difference between the old and the new reachability}\label{minnegy}
\par
In this section we prove that we can find a reduction method for all
graphs and for each pebbling distribution if the graph contains a $P_3
\square P_2$ subgraph which has at least four pebbles. Usually we fix
the distribution $p$ and consider subgraph $R=P_3 \square P_2$ which
has the most pebbles. So if $R'$ is also a $P_3\square P_2$ subgraph
of $G$, then $p(R)\geq p(R')$.

First, we show how can we prove the solvability of a reduced
distribution by calculation. We next show how to do this is one of the 
cases. All other ones are handled in a similar way. These calculations
are contained in the appendix.

Let $p$ be a solvable pebbling distribution which satisfies
$p(\overline{l})+p(\overline{x})\geq 4$. A proper reduction method in
this case is the following: Take the pebbles from vertices
$\overline{l}$ and $\overline{x}$, throw away two of these pebbles and
place the remaining ones to $\overline{A}$. Place the other pebbles of
$R$ at vertices of rung $A$ and $B$ as shown in Fig. \ref{abra}.

\begin{figure}[htb]
\centering
\scalebox{0.6}{\input{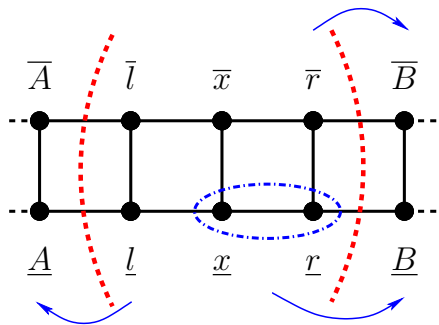}}
\caption{An example for a proper reduction method. We move the pebbles of $\underline{l}$ to $\underline{A}$, the pebbles of $\overline{r}$ to $\overline{B}$ and finally the pebbles of $\underline{x}$ and $\underline{r}$ to $\underline{B}$.}
\label{abra}
\end{figure}        
$p^R$ is not uniquely defined, but we can show that $p^R$ will be always solvable. 
We prove this by the following  calculations: 

\begin{align*}
Rw_{p^R}(\overline{A})&-Rw_{p}(\overline{A})=(p(\overline{A})+p(\overline{l})+p({\overline{x}})-2+\frac{1}{2}(p(\underline{A})+p(\overline{r})+p(\underline{l}))+\\
&+\frac{1}{4}(p(\underline{x})+p(\underline{r}))+\underbrace{\frac{1}{2}p(\overline{B})+\frac{1}{4}p(\underline{B})+\cdots}_{\Delta_1})-\\
&-(p(\overline{A})+\frac{1}{2}(p(\underline{B})+p(\overline{l}))+\frac{1}{4}(p(\underline{l})+p(\overline{x}))+\frac{1}{8}(p(\underline{x})+p(\overline{r}))+\\
&+\frac{1}{16}p(\underline{r})+\underbrace{\frac{1}{16}p(\overline{B})+\frac{1}{32}p(\underline{B})+\cdots}_{\Delta_2})=\\
&=\underbrace{\frac{1}{2}p(\overline{l})+\frac{3}{4}p(\overline{x})-2}_{\geq 0}+\underbrace{\frac{3}{8}p(\overline{r})+\frac{1}{4}p(\underline{l})+\frac{1}{8}p(\underline{x})+\frac{3}{16}p(\underline{r})}_{\geq 0} +\underbrace{\Delta_1-\Delta_2}_{\Delta}\geq 0
\end{align*}
We assumed that $p(\overline{l})+p(\overline{x})\ge 4$,  which implies $\frac{1}{2}p(\overline{l})+\frac{3}{4}p(\overline{x})-2\ge \frac{1}{2}p(\overline{l})+\frac{1}{2}p(\overline{x})-2\ge 0$.
$\Delta\geq 0$ because each vertex and its pebbles right from $B$ come closer to $\overline{A}$, even if its rung have been reflected.   
\par
Similar calculations are needed for $\overline{B}, \underline{A}$ and $\underline{B}$. The details are left to the reader, we only give here the crucial parts:
$$Lw_{p^R}(\overline{B})-Lw_{p}(\overline{B})=\underbrace{\frac{3}{8}p(\overline{l})+\frac{1}{4}p(\overline{x})}_{\geq 1}-1+\mbox{nonnegative}+\Delta\geq 0,$$
$$Rw_{p^R}(\underline{A})-Rw_{p}(\underline{A})=\underbrace{\frac{1}{4}p(\overline{l})+\frac{3}{8}p(\overline{x})}_{\geq1}-1+\mbox{nonnegative}+\Delta\geq 0,$$
$$Lw_{p^R}(\underline{B})-Lw_{p}(\underline{B})=\underbrace{\frac{3}{16}p(\overline{l})+\frac{1}{8}p(\overline{x})}_{\geq\frac{1}{2}}-\frac{1}{2}+\mbox{nonnegative}+\Delta\geq 0.$$
This implies that this is a proper reduction method.

\par
There are 20 essentially different ways to place at least 4 pebbles to $R$. In most of these cases one can find a proper reduction method and a similar argument to verify it. For completeness these are listed in the appendix.
Unfortunately, a universal reducing method which is proper for every pebbling distribution does not exist. 
The cases where a proper reducing method does not exist can be seen in Fig. \ref{kivetelek}.

\begin{figure}[htb]
\centering
\scalebox{0.6}{\input{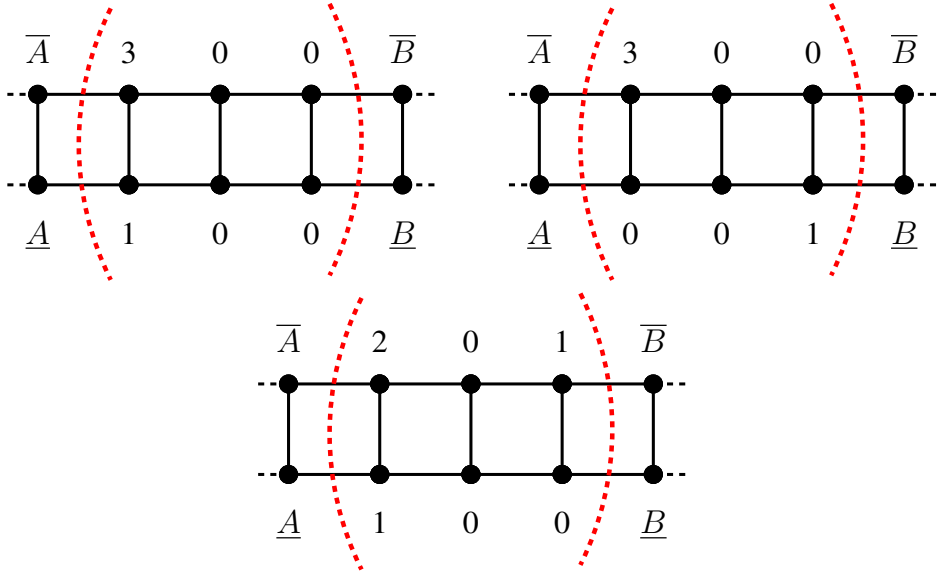}}
\caption{Exceptions when the modified inequalities do not hold.}
\label{kivetelek}
\end{figure}

The first case can be solved easily. If $R$ is at the left end of the graph then put two pebbles at the vertices of rung $B$. Otherwise, we can choose another subgraph $R'$ which contains rung $A$, $l$ and $x$ and it is not an exception, so we have got a proper reducing method for it. We call this idea the \textit{shifting technique}.\par

In the second case, there is no reduction method for this $R$ unless we use the assumption 
 that  $R$  contains the maximum number of pebbles from the set of $P_3\square P_2$ subgraphs. 
 So  assume that each of the $P_3\square P_2$ subgraphs contain at most four pebbles. The reducing method is the following: Put one pebble onto $\overline{A}$ and the other onto $\underline{B}$. It is clear that the vertices of rung $A$ and $B$ are reachable. The modified inequalities can be shown to hold for vertices $\underline{A}$, $\overline{B}$, $\underline{B}$, hence the original inequalities hold for these vertices, too. Now we need to show that $R_p(\overline{A})=1$, which implies that the original inequality holds for vertex $\overline{A}$. 

\par     
Partition $G$ to disjoint $P_3\square P_2$ subgraphs. The maximality of $R$ means that these subgraphs contain at most four pebbles. This gives an upper bound for $R_p(\overline{A})$. Thus
$$R_p(\overline{A})\leq \left\lfloor  \frac{3}{2}+\frac{1}{16}+\frac{4}{16}\sum_{i=0}^{\infty}\left(\frac{1}{8}\right )^i\right\rfloor=\left \lfloor 1+\frac{9}{16}+\frac{2}{7}\right \rfloor=1.$$
Therefore the original inequalities hold, hence this reduction method is proper.

\par Finally, consider the third case and put a pebble onto $\underline{A}$ and $\overline{B}$. The modified inequalities hold for $\underline{A}$, $\overline{B}$ and $\underline{B}$. We do what we have done in the second case. We can make an estimation again:
$$R_p(\overline{A})\leq \left\lfloor  1+\frac{1}{4}+\frac{1}{8}+\frac{4}{16}\sum_{i=0}^{\infty}\left(\frac{1}{8}\right )^i\right\rfloor=\left \lfloor 1+\frac{3}{8}+\frac{2}{7}\right \rfloor=1.$$ 
So the original inequalities also hold.
\par
We have shown that there exist a solvable reduced distribution for every distribution which has a $P_3\square P_2$ subgraph that contains at least four vertices.

\subsection{Distributions where the maximal number of pebbles on a $P_3\square P_2$ subgraph is three}\label{maxharom}

In this case there are several constraints on $p$. These lead us to some useful observations. The most important is that none of the rungs can get more than three pebbles from a fixed direction. The next lemma formulates this. 

\begin{lemma}
Let $p$ be a distribution and fix a subgraph $R$. Let T be an executable rubbling sequence which uses vertices located to the left from $R$. If every $P_3\square P_2$ subgraph has at most three pebbles then $p_T(\overline{A})+p_T(\underline{A})\leq 3$. Furthermore, if equality holds then $p(A)=3$. 
\label{szummalemma}
\end{lemma}

\begin{proof}
The proof is based on an idea similar to the one we used in the previous subsection. Partition the graph to disjoint $P_3\square P_2$ subgraphs. By the assumption, all of these disjoint subgraphs may contain at most 3 pebbles. When $p(A)=3$ we obtain the following estimate which completes the proof of this case:
$$p_T(\overline{A})+p_T(\underline{A})\leq \left\lfloor 3+\frac{3}{8}\sum_{i=0}^{\infty}\left(\frac{1}{8}\right )^i\right \rfloor=\left \lfloor 3+\frac{3}{7} \right \rfloor =3.$$
When $p(A)\leq 2$, then the third pebble of the subgraph is not on $A$, therefore its contribution is at most $\frac12$:
$$p_T(\overline{A})+p_T(\underline{A})\leq \left\lfloor 2+\frac{1}{2}+\frac{3}{8}\sum_{i=0}^{\infty}\left(\frac{1}{8}\right )^i\right \rfloor=\left \lfloor 2+\frac{1}{2}+\frac{3}{7} \right \rfloor =2.$$
\end{proof}

To continue the proof we have to find reduction methods for the cases when $R$ contains three pebbles. There are sixteen different cases, these are shown on Fig. \ref{3asok} with proper reduction methods.

\begin{figure}
\includegraphics[scale=0.45]{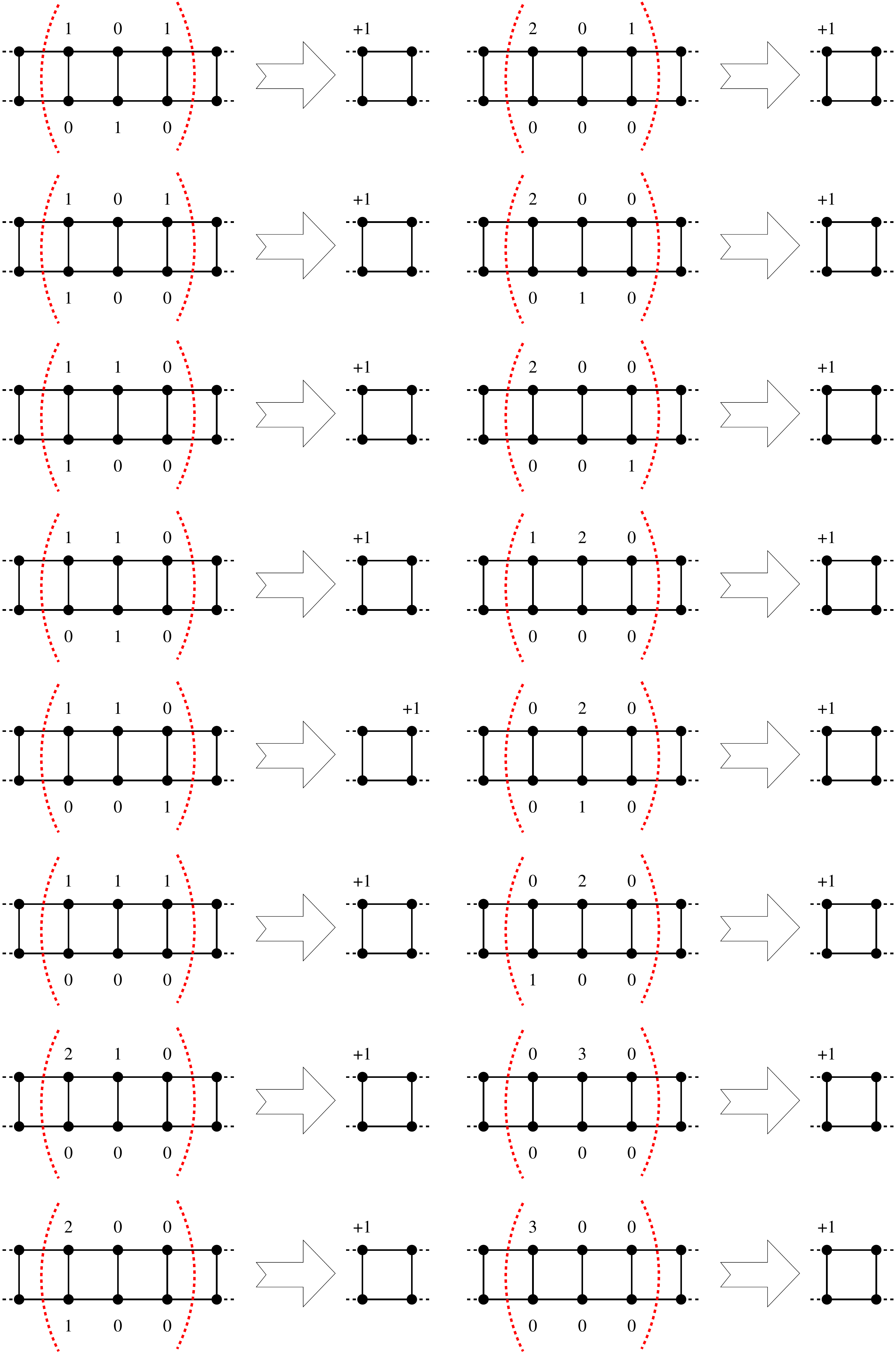}
\caption{Reduction methods when every $P_3\square P_2$ contains at most three pebbles.}
\label{3asok}
\end{figure}

 We prove the solvability of the reduced pebbling distribution for one case, and leave the remaining ones for the reader to check. It is enough to check that the right (left) reachability of the vertices of rung $A$ ($B$) are not decreasing, and all of them remain reachable.

\par Consider the case $p(\overline{l})=p(\underline{x})=p(\overline{r} )=1$ and $p(\underline{l})=p(\overline{x})=p(\underline{r})=0$. A proper reducing method is the following: Place a pebble at $\overline{A}$. 
Since in this case any $P_3\square P_2$ subgraph contains at most 3 pebbles, we have that $p(A)\leq 1$ and $p(B)\leq 1$ (otherwise the $P_3\square P_2$ subgraph shifted one to the left or to the right would contain $\ge 4$ pebbles). Lemma \ref{szummalemma} shows that $A$ is not left 3-reachable under $p$ (it can not have 3 pebbles under $p$). 
This result shows that rung $r$ is not left 2-reachable under $p$. Hence a vertex of $B$ can get only one pebble from $r$ by a strict rubbling move and not left 2-reachable. Similar statement holds for rungs $l$ and $A$ when we swap left and right directions. Now we can show that either there is a pebble on $B$ or its right neighbours are right reachable. 

 To show this, assume the contrary, so $p(B)=0$ and one of its right neighbour, $u$, is not right reachable. Then the other right neighbour cannot be right 2-reachable. This means that we can not reach the vertex of $B$ which is adjacent to $u$ without the use of the other vertex of $B$. But to move there a pebble, we consume all the pebbles which can be moved to the neighbourhood of $B$. So this vertex of $B$ cannot be reached. 

\par Moreover, the vertices of $B$ can be left reachable if and only if $p(B)=1$. There is also a similar fact for $A$. Now we need to check the reachability of the four vertices.
\begin{itemize}
\item $\overline{A}$: We place a pebble at this vertex, so it is right reachable.
\item $\underline{A}$: $l$ is not right 2-reachable, hence the extra pebble at $\overline{A}$ can act as a pebble of $\underline{l}$ when we want to reach $\underline{A}$ from $p^R$. 
\item $\overline{B}$:  If it is left reachable under $p$ then it is also left reachable under $p^R$ with the help of $\overline{A}$'s pebble. Otherwise this pebble assures reachability of $\overline{B}$ under $p^R$.
\item $\underline{B}$: $R_p(\underline{A})=0$, thus the addition of an extra pebble at $\overline{A}$ makes $\underline{A}$ left reachable. So if $B$ has a pebble then $\underline{B}$ is left reachable under $p^R$, otherwise simply reachable.   
\end{itemize} 
\par The proof of the solvability of the reduced distributions in the first 14 cases uses same tools and ideas. We can reduce Case 15 and Case 16 to Case 10 and Case 14 with the shifting technique.     

\subsection{Distributions where none of the $P_3\square P_2$ subgraphs contain more than two pebbles } \label{maxketto}

In this subsection we show that if $p$ is solvable and none of the $P_3\square P_2$ subgraphs contain more than two pebbles then all of them contain exactly two. 

\begin{lemma}
Let $p$ be a distribution such that every $P_3\square P_2$ subgraph contains at most 2 pebbles. If there is an $R=P_3\square P_2$ subgraph which satisfies $p(R)<2$, then $p$ is not solvable.
\label{keves}
\end{lemma}  

\begin{prop}
Let $p$ be a distribution such that every $P_3\square P_2$ subgraph contains at most 2 pebbles. A rung $g$ is 2-reachable from $p$ if and only if $p(g)=2$. If a rung is 2-reachable then it can not get a pebble by a rubbling move.
\label{keves2}
\end{prop}

It is easy to show the proposition with the partition method described in the proof of Lemma \ref{szummalemma}.

\begin{proofof}{Lemma \ref{keves}}
Let $R$ be a $P_3\square P_2$ subgraph satisfying $p(R)\leq 1$.  There are three essentially different cases (considering symmetry):
\begin{itemize}
\item $p(R)=0$: Rung $A$ and rung $B$ contains at most two pebbles, hence one of the vertices of rung $x$ is not reachable.
\item $p(x)=1$: The upper bound on $p(P_3\square P_2)$ implies that $p(A)$ and $p(B)$ is less than or equal to one. We can not move an extra pebble to rung $A$ and $B$ due to the previous proposition. Hence the vertex of rung $x$ which does not have a pebble is not reachable.
\item $p(l)=1$: Clearly $B$ neither contains 3 pebbles nor can get an additional pebble. Furthermore $A$ is not 2-reachable, thus one of the vertices of rung $x$ is not reachable again.   
\end{itemize}
\end{proofof}


Now we are prepared to complete the proof of the main result.


\begin{proofof}{Theorem \ref{letratetel}}\label{fobiz}
	Let $G$ be a counterexample where $V(G)$ is minimal, and let $p$ be the optimal distribution on $G$.
If there exists a subgraph $R=P_3\square P_2$ such that $p(R)\geq 3$ then we can apply one of the reduction methods described in the previous section and get a solvable distribution $p^R$ on graph $G^R$. $|V(G^R)|=|V(G)|-6$, and $\varrho_{\opt}(G^R)\leq p^R(G^R)=p(G)-2=\varrho_{\opt}(G)-2$. Thus  $G^R$ is also a counterexample, which contradicts the minimality of $G$. So we can assume that every $P_3\square P_2$ subgraph contains at most two pebbles. By Lemma \ref{keves} every $P_3\square P_2$ subgraph contains exactly two pebbles. The solvability of $p$ requires that the pattern has to start and end with two 1s or one 2 because of Proposition \ref{keves2}. Thus  number of pebbles on the rungs must have the pattern shown on Fig. \ref{pattern}. 
\begin{figure}[htb]
\centering
\includegraphics[scale=0.6]{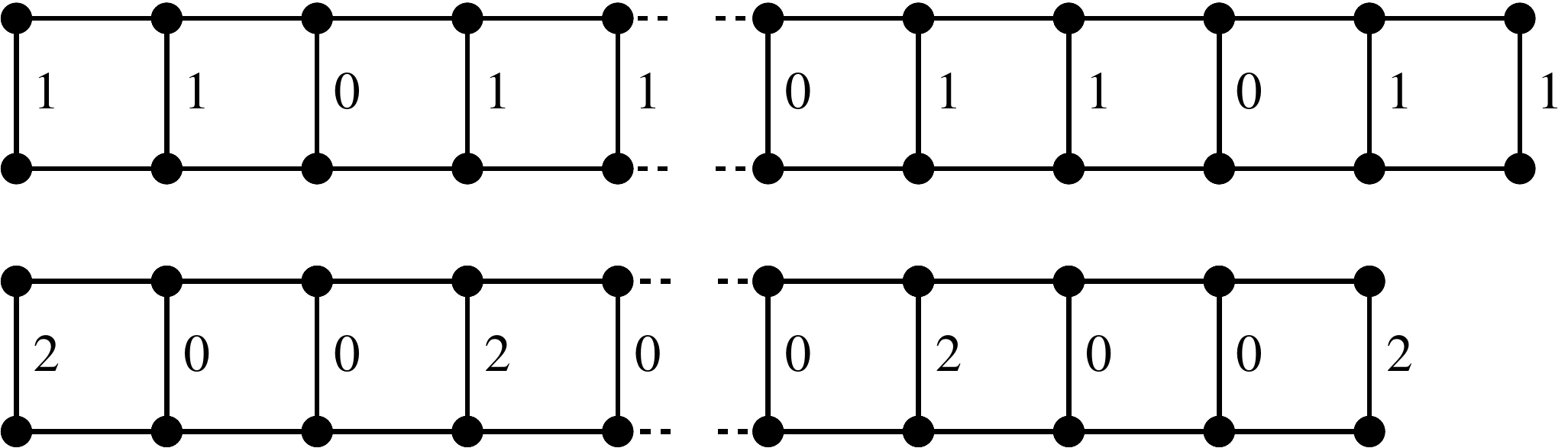}
\label{pattern}
\caption{Solvable distributions where each $P_3\square P_2$ subgraph contains exactly two pebbles, have to have one of these patterns. The numbers show that how many pebbles placed on the rungs.}
\end{figure}
 However, this means that $G$ is not a counterexample. 

Thus we proved the  lower bound, hence together with the upper bound (Lemma \ref{letratetelfelso}) the proof of  Theorem \ref{letratetel} is complete. 
 \end{proofof}

\section {The 2-optimal rubbling number of the circle}

In this section the 2-optimal rubbling number of the cycle is
determined. It is interesting on its own, but it will be needed in the
next section.

To do this we are going to use \emph{smoothing}. Smoothing is 
well known technique to determine optimal pebbling number. It was invented
in \cite{ladder}, and its modification \emph{rolling} is also used for
optimal rubbling in \cite{BelSie}. Smoothing utilizes the following
observation: It is not worth  putting many pebbles at a vertex of degree 2.

\begin{definition}
Let $p$ be a pebbling distribution on the graph $G$. Let $v$ be a vertex of degree 2 such that $p(v)\geq 3$. A \emph{smoothing move} from $v$ removes two pebbles from $v$ and adds one pebble at both neighbours of $v$.
\end{definition}

\begin{lemma}\cite{ladder}
Let $p$ be a pebbling distribution on $G$.  Let $u$ and $v$ be different vertices of $G$ such that $u$ is $k$-reachable and $p(v)\geq 3$. After the application of a smoothing move from vertex $v$, $v$ remains reachable and $u$ remains $k$-reachable.
\end{lemma}

This lemma is used for pebbling, but it works for rubbling, too. Also, we can state something stronger in case of rubbling: 
\begin{lemma} If $p(v)\geq 3$, $d(v)=2$ and $u$ is 2-reachable under $p$, then $u$ is 2-reachable under $q$ which is obtained from $p$ by making a smoothing move from $v$. 
\end{lemma}
\begin{proof}
The only part which is not straightforward from the smoothing lemma is why $v$ is 2-reachable. In turn, this is easy: Both neighbours of $v$ have a pebble, hence we can move a pebble to $v$ by a strict rubbling move. Furthermore $v$ had a pebble before the rubbling move. 
\end{proof}
\begin{theorem}
The $2$-optimal rubbling number of $C_n$ is $n$. 
\label{kor2tetel}
\end{theorem}

Clearly, if $p(v)\geq 3$, $d(v)=2$ and $u$ is 2-reachable under $p$, then $u$ is 2-reachable under $q$ which is obtained from $p$ by making a smoothing move from $v$. When no smoothing move is available, we say that the distribution is smooth.

\begin{proofof}{Theorem \ref{kor2tetel}}
Consider vertices $v$ and $w$ of the circle. There are exactly two paths in the circle which have these two vertices as end vertices. When every inner vertex of one of these paths is occupied, we say that $v$ and $w$ are \textit{friends}. 
Assume that $\varrho_{\text{2-opt}}(C_n)<n$, so there is a 2-solvable distribution $p$ with size $n-1$. Apply smoothing moves repeatedly for every vertex which contains at least three pebbles. $p$ has less than $n$ vertices, thus this process ends in finitely many steps. Now we have a 2-solvable pebbling distribution $q$ such that $q(v)\leq 2$ for every vertex $v$. Denote the number of vertices containing $i$ pebbles with $x_i$. We have the following two equalities:
$$x_2+x_1+x_0=n,$$
$$2x_2+x_1=n-1,$$  
which imply that $x_0=x_2+1$. We also know that
$q$ is solvable and $|q|=n-1$.  Now we  show that it has the following property:
\begin{lemma}
An unoccupied vertex in distribution $q$ must have two different friends such that each of them contains two pebbles.
\end{lemma}


\begin{proof}
Let $v$ be an unoccupied vertex which does not have two friends having two pebbles. 

If there are no more unoccupied vertices then $q$ is the distribution where there is an unoccupied vertex, and any other vertex has one pebble. It is clearly not 2-solvable, proving our claim in this case.

If there are exactly two unoccupied vertices, then 
there is only one vertex in the graph which has two pebbles. 
If we can move two pebbles to a vertex of degree two then we can move one of them by a pebbling move. 
A pebbling move requires that a vertex has at least two pebbles. 
It is straightforward that it is not worth moving a pebble to a vertex of  degree two  by a strict rubbling move (except if we move it to $v$), so we can only use pebbling moves except at the last move. 
So, to  move 2 pebbles to $v$, we need to pass the second pebble of the vertex having 2 pebbles to $v$ by a sequence of pebbling moves. But after we apply these pebbling moves, the neighbour of $v$ which was involved in these moves is unoccupied. No more pebbling moves are available hence we can not reach this vertex without using  the pebble of $v$. We still need to move one more pebble to $v$ but we can not do it by either a pebbling move or a strict rubbling move.
This shows that $v$ is not 2-reachable.  

Finally consider the case when there are at least three unoccupied vertices. We show that a similar contradiction will arise like in the previous case. Let $u$ and $w$ be the two unoccupied friends of $v$. These three vertices cut the graph to three paths. Denote them $P_1$, $P_2$, $P_3$, such that $v$ is between $P_1$ and $P_2$. Thus the two endvertices of $P_3$ are $u$ and $w$. Either $P_1$ or $P_2$ does not contain a vertex having 2 pebbles, we can assume it is  $P_1$, whose ends are   $v$ and $u$. Also $P_2$ contains at most one vertex having 2 pebbles. (All such vertices would be friends of $v$). 

The two ends of $P_3$ are unoccupied, all other vertices contain $\le 2$ pebbles. By the weight argument at most $\lfloor\sum_{i=1}^{n}\frac{2}{2^i}\rfloor=1$ pebble can be moved to the ends of $P_3$. However, this pebble cannot be used in any further pebbling moves, so the number of pebbles inside $P_1$ and $P_2$ cannot increase. Therefore the situation is similar to the previous case. It is possible to move 1 pebble to $v$ using the pebbles inside $P_2$, but it is not possible to move another one using pebbles of $P_1$.
\end{proof}


Each unoccupied vertex has at least two friends who have two pebbles, but each vertex having 2 pebbles has at most two unoccupied friends. Thus the number of unoccupied vertices cannot be larger than the number of the vertices having 2 pebbles. This contradicts $x_0=x_2+1$, hence every 2-solvable pebbling distribution contains at least $n$ pebbles. 

On the other hand, $\varrho_{\text{2-opt}}(C_n)\le n$, because we can place 1 pebble at each vertex to obtain a 2-solvable distribution.
\end{proofof}

\section{Optimal rubbling number of the $n$-prism}
\begin{prop}
If $k\ge 2$ then
$$\varrho_{\opt}(C_{3k-1} \square P_2)\leq\varrho_{\opt}(P_{3k-2}\square P_2)\leq 2k,$$
$$\varrho_{\opt}(C_{3k} \square P_2)\leq\varrho_{\opt}(P_{3k-1} \square P_2)\leq 2k,$$
$$\varrho_{\opt}(C_{3k+1} \square P_2)\leq\varrho_{\opt}(P_{3k} \square P_2)\leq 2k+1.$$

\label{all}

\end{prop}

\begin{proof}
Notice that if  an arbitrary rung $A$ of $C_{n}\square P_2$ is deleted then we obtain $P_{n-1}\square P_2$.
It is easy to see that if $k\geq 2$ then opposite ends of the ladder are reachable ``in parallel'' from the distributions shown in Fig.~\ref{felso} such that every pebble contributes to only one end. Since $A$ is adjacent to both ends of the ladder, both vertex of $A$ can be reached, too. In one case this idea works even if we delete two adjacent rungs from the prism. So the optimal distributions of $P_{n-1}\square P_2$ or $P_{n-2}\square P_2$ gives a solvable distributions of the circle in the following way:
\begin{itemize}
\item If $n\equiv 0 \mbox{ mod } 6$ then use the distribution of $P_{3k+2}\square P_2$.
\item If $n\equiv 3 \mbox{ mod } 6$ then use the distribution of $P_{3k+1}\square P_2$.
\item If $n\equiv 1 \mbox{ mod } 3$ then use the distribution of $P_{3k}\square P_2$.
\item If $n\equiv 2 \mbox{ mod } 3$ then use the distribution of $P_{3k+1}\square P_2$.
\end{itemize}
\end{proof}

For the optimal distributions when $k=1$ see Fig. \ref{C_nxP_2}\protect\footnotemark.

\footnotetext{This result is going to appear in Discrete Applied Mathematics. Unfortunately, there is an error in that version. It states that the optimal distribution of the cube ($C_3\square P_2$) contains $4$ pebbles, but the right value is $3$. This error slightly changes the statement of Theorem \ref{Mobiustetel}.}

\begin{figure}[htb]
\centering
\includegraphics[scale=0.6]{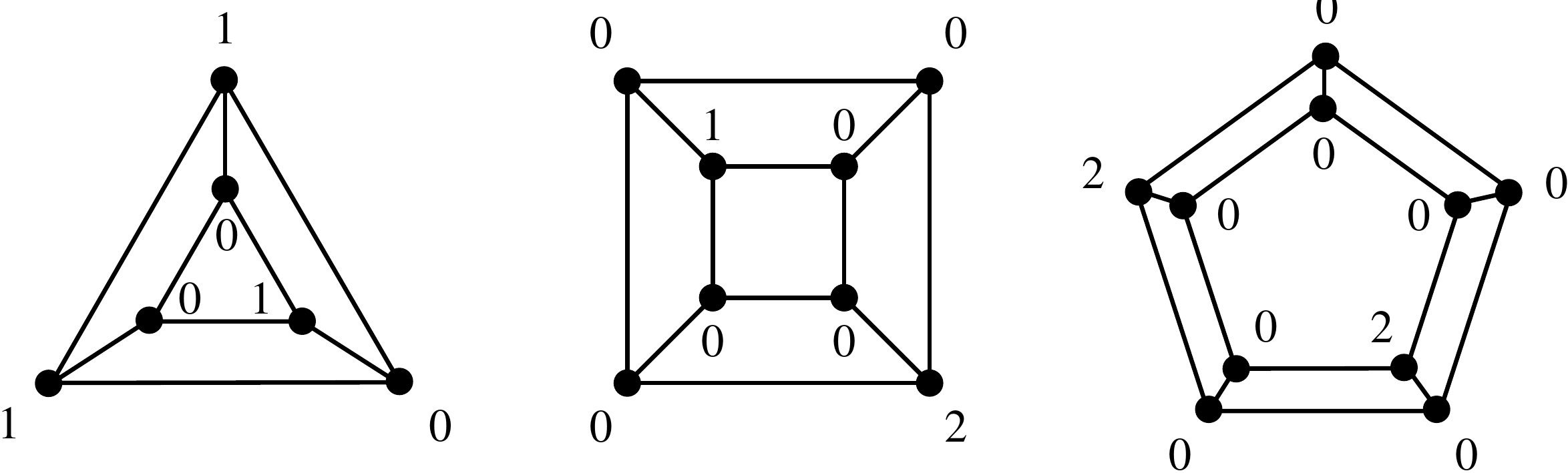}
\caption{Optimal distributions of the n-prism when $n\leq5$.}
\label{C_nxP_2}
\end{figure}

To prove the lower bound we show the following: If $p$ is an optimal distribution of $C_n\square P_2$, then we can delete one or two adjacent rungs and slightly modify the distribution such that the obtained distribution is solvable. Thus we get a lower bound by \ref{letratetel}. We will show this using the assumption that one of the deleted rungs is not 2-reachable. Therefore we start with Lemma \ref{nem2elerheto}, which guarantees that such a rung is exists.    

\begin{lemma}
Let $p$ be an optimal distribution of $C_n\square P_2$ ($n\geq 5$). Then there exists a rung which is not 2-reachable.
\label{nem2elerheto}
\end{lemma}

We are going to use the \emph{collapsing technique} of \cite{ladder}, which is the following method. Let $S$ be a subset of $V(G)$. Then the collapsing  $S$ creates a new graph $H$. The vertex set of $H$ is $\{u\}\cup V(G)\setminus S$,  where $u$ is a vertex which plays the role of $S$. More precisely, $u$ is connected with a vertex $v$ if $v$ and a vertex of $S$ are adjacent in $G$. Moreover, the subgraph induced by $V(G)\backslash S$ is the same in both graphs. Let $p$ be a pebbling distribution on $G$. Then we also define a \textit{collapsed} distribution $q$ on the collapsed graph $H$. The definition is simple: $q(u)=p(S)$ and $q(v)=p(v)$ for all $v\notin S$.      

\begin{proof}
Let $q$ be a pebble distribution of $C_n$ obtained from a solvable pebble distribution of $C_n\square P_2$ by applying collapsing operations for each rungs independently. It is easy to see, that the 2-reachability of a rung implies that the vertex produced by its collapse is 2-reachable from $q$. Assume that each rung is 2-reachable from $p$. Then each vertex is 2-reachable from $q$ and Lemma \ref{kor2tetel} implies that $q(C_n)\geq n$.  However, Proposition \ref{all} implies that $p(C_n\square P_2)<n$ and the definition of the collapsed distribution shows us that $q(C_n)=p(C_n\square P_2)$, which is a contradiction.
\end{proof}

We stated that the optimal distributions of $P_n\square P_2$ which we inspected has the property that the opposite ends of the ladder are reachable in ``in parallel''. Unfortunately, it is not true for all solvable distributions. To handle this, we invent a new definition which catches an even stronger property.

\begin{definition}
Let $p$ be a distribution of $P_n\square P_2$. Let $l$ and $r$ be different vertices, such that $l$ is located left from $r$. If $l$ is right $k_l$-reachable and $r$ is left $k_r$-reachable, but not independently, then we say that $l$ and $r$ are $p$-dependent.
\end{definition}

\begin{lemma}
\label{nemfugg}
 Let $p$ be a distribution of $P_n\square P_2$. Let $l$ and $r$ be two different vertices, such that they belong to different rungs. If $l$ and $r$ are $p$-dependent, then all vertices located between their rungs are reachable.
\end{lemma}

\begin{proof}
Without loss of generality we may assume that $l$ is located left from $r$.
The condition implies the following facts: 
\begin{itemize}
\item There is a rubbling sequence $T_r$ acting only on vertices not located to the right from $R$ (the rung containing $r$), such $p_{T_r}(r)=k_r$. 
\item If we consider a proper subsequence of $T'_r$ (i. e. we delete at least one move from $T_r$) then $p_{T'_r}(r)< k_r$ or $T'_{r}$ is not executable. 
\item We also have a $T_l$ with the same properties for $l$. 
\item There is a vertex $u$, such $T_r$ and $T_l$ both acts on this vertex.
\end{itemize}

\par 
Let $v$ be a vertex between the rungs of $l$ and $r$. We show that $v$ is reachable. Let $V$ be the rung which contains $v$. It is easy to see, that at least one of $T_r$ and $T_l$ is acting on $V$. Assume that it is $T_r$. If $T_r$ acts on $v$ then it has to receive a pebble at some point, so it is reachable. Otherwise, the other vertex of $V$ (denote it with $v'$) is reachable. $T_r$ moves the pebble of $v'$ towards $r$, there are two possibilities. First, when  $T_r$ uses a pebbling move to move this pebble from $v'$. In this case, we can change its destination vertex to $v$ and so it is reachable. The second case is when $T_r$ is using a strict rubbling move to move the pebble of $v'$ towards a neighbour. The structure of the ladder implies that this strict rubbling move removes the other pebble from an other neighbour of $v$. So we can reach $v$ by this strict rubbling move if we change its destination to $v$.
\end{proof}

We prove two lower bounds on $\varrho_{\opt}(C_n\square P_2)$ in the next lemma. One of them is stronger than the other. However, the stronger one is not always true, but we give a sufficient condition for it. 

\begin{lemma}
\label{korlatok}
Let $n\geq 5$ and $p$ be an optimal distribution of $C_n\square P_2$, and let $C$ be a rung of this graph such that it is not 2-reachable under $p$. Consider the following properties:
\begin{itemize}
\item $p(C)=0$
\item $L_p(\overline{L})=0$
\item $L_p(\underline{L})=1$
\item $R_p(\overline{R}\geq 2)$ or $R_p(\underline{R}\geq 2)$
\end{itemize}  
Let $\mbox{ref}_c()$ be a reflecting operator, which reflects the whole $p$ distribution across rung $C$. Furthermore, let $\mbox{ref}_h()$ be a similar operator which reflects the whole $p$ distribution horizontally.   
If none of the $\{p$,
$\mbox{ref}_c(p)$,
$\mbox{ref}_v(p)$, $\mbox{ref}_c(\mbox{ref}_v(p))\}$ distributions fulfills all of the properties given above, then $|p|\geq \varrho(P_{n-1}\square P_2)$. Otherwise $|p|\geq \varrho(P_{n-2}\square P_2)$.  
\end{lemma}

\begin{proof} We construct
 solvable pebbling distributions for the $P_{n-1}\square P_2$ and $P_{n-2}\square P_2$ ladders from the optimal distribution of $C_{n}\square P_2$.
Consider the solvable distribution $p$ on $C_{n}\square P_2$ and the rung $C$ which is not 2-reachable.
It is trivial that each of the distributions  $\{\mbox{ref}_c(p)$,
$\mbox{ref}_v(p)$, $\mbox{ref}_c(\mbox{ref}_v(p))\}$ are solvable. If we delete the rung  $L$ or rung $C$ from $C_{n}\square P_2$, then the remaining graph is $P_{n-1}\square P_2$. If we delete both  $L$ and $C$ then the remaining graph is $P_{n-2}\square P_2$. We show that a slight modification of the induced distribution is a solvable distribution of these graphs.

This modification is the following: Let $L_2$ be the left neighbour of $L$. If we delete $L$, then we place all  pebbles of $\underline{L}$ to $\overline{L_2}$, and place all the pebbles of $\overline{L}$ to $\underline{L_2}$. 

\par $C$ is not 2-reachable, thus  no pebbling move can move a pebble from $C$ to $L$. 
Thus only a strict rubbling move can move a pebble from $C$ to $L$.
The modification does not increase the distance of the pebbles from the vertices located left from $L$.
Furthermore, the strict rubbling move from $C$ is not needed, because the result of this move is the same as if we swap the pebbles of the vertices of $L$, but in the modified distribution these pebbles are placed closer to the remaining vertices. Hence if $T$ is an executable rubbling sequence  acting on $L_2,L$ and $C$ then we can construct $T'$,  such it is also executable, $p_T'(\overline{L_2})=p_T(\overline{L_2}) $ and $p_T'(\underline{L_2})=p_T(\underline{L_2}) $. The construction is easy: delete the strict rubbling move which uses a pebble from $C$, replace each occurrence of $\overline{L}$ with $\underline{L_2}$, and do the same with the $\underline{L}$, $\overline{L_2}$ pair. 

\par Using this we can modify every executable rubbling sequence which acts on $L$.
     
\par  Notice that if we delete the edges between $L$ and $C$, and  we have a vertex $l$ of $L$ and an $r$ of $R$ such they are $p$-dependent
, then by Lemma \ref{nemfugg} all vertices between $L$ and $R$ are still reachable.

\subsubsection*{Case 1:}
First we handle the case when there is a rubbling sequence from the optimal distribution to each vertex of $C$ that non of them use either the two edges between $L$, $C$ or the two edges between $C$, $R$. We can assume that we have the first case, otherwise apply $\mbox{ref}_c()$. Now we delete $L$ and show that the remaining part is solvable under a modified version of $p$. 

If we delete the edges between $L$ and $C$ and a vertex of $L$ and a vertex of $R$ is $p$-dependent, then all vertices between $L$ and $R$ are reachable. Furthermore $C$ is reachable our assumption, hence $R$ is also. Therefore, the modified distribution  is  a solvable distribution in the graph obtained by deleting these edges, a $P_{n-1}\square P_2$.   Thus we may assume that the vertices of $L$ and $R$ are not $p$-dependent.

First we show, that there is no need to move a pebble from $L$ across an edge between $L$ and $C$ to reach a vertex located right from $C$. A vertex of $L$ cannot get two pebbles from the left, because it would violate the condition that $C$ is not 
2-reachable ($C$ remains reachable from $R$ since $l\in V(L)$ and $r\in V(R)$ are not $p$-dependent ). So there is only a strict rubbling move available which requires a pebble at $C$.

\par If $C$ does not have a pebble then first we need to move one there. This requires two pebbles at $R$ and after two other moves  we are able to move it back to $R$ with the help of a pebble of $L$. Both vertices of $R$ were reachable, we consumed two pebbles of $R$ and moved one to it, so we just wasted the pebbles. Hence it is pointless to move a pebble from $L$ to $R$.

\par  If $C$ has a pebble then we may assume that $\overline{C}$ has it. We can apply $\mbox{ref}_h()$ to achieve this. Now we can use only the pebble of $\underline{L}$. Our assumption that both vertices of $C$ are reachable without the use of the edges between $L$ and $C$ implies that $\underline{R}$ is right-reachable. Hence \underline{L} can not be left-reachable. 

\par This completes the proof of this case. 

\subsubsection*{Case 2:}
Now we assume the opposite, that there are moves through the edges between $L$, $C$ and $R$ to reach the vertices of $C$.

\par When $C$ has a pebble, then assume that it is placed on $\overline{C}$. The reachability of $\underline{C}$ implies that $\underline{R}$ is right reachable or $\underline{L}$ is left reachable. However these are the cases we covered in Case 1. Hence $C$ can not contain a pebble. 

\par If there is a left 2-reachable $l$ of $L$ and a right 2-reachable $r$ of $R$ then we have two possibilities. The first is that they are $p$-dependent. This means that if we delete $C$ then $p$ remains solvable on the graph because of Lemma \ref{nemfugg}. The second case, when they are not $p$-dependent, contradicts with the fact that $C$ is not 2-reachable.

\par If both vertices of $L$ are left reachable and both vertices of $R$ are right reachable but none of them are left or right 2-reachable, then it is easy to see that we can not move a pebble from $L$ to $R$ through $C$ and the same is true for the opposite direction. So in this case  and everything remains reachable after the deletion of $C$.  

\par It is possible that one of them can get more pebbles, so assume that $\overline{R}$ is right 2-reachable. In this case $\underline{R}$ and $\overline{R}$ are not reachable independently. The same is true for $\overline{L}$ and $\underline{L}$. Hence after we have used the two pebbles of $\overline{R}$ and have moved it to $\overline{C}$, we can not use the pebble at $\overline{C}$ by a strict rubbling move to increase the number of pebbles at $\overline{L}$ to 2.  
So we can delete $C$ without any problem and the distribution remains solvable. 

\par Now we have checked almost all the cases except the ones which gives bound with $P_{n-2}\square P_2$. In this case at most 3 of the vertices of $L$ and $R$ are reachable from the proper direction. Otherwise, we can not reach both vertices of $C$ or we have a case which we have already covered above. 

Applying $\mbox{ref}_c()$ and $\mbox{ref}_h()$ we can guarantee that $\overline{L}$ is not left-reachable.
Then $\underline{L}$ is left reachable, and both vertices of $R$ are right reachable. The fact that $\overline{L}$ is not left reachable implies that we can move a pebble from $R$ to $\overline{C}$ or to $\underline{L}$. $C$ does not have a pebble so $R_p(R)\geq 2$. Now if we delete $L$, $C$ and modify the distribution in the above way then we obtain a solvable distribution of $P_{n-2}\square P_2$ which gives us the desired bound.

It remains to show that  $R_p(\overline{R})=R_p(\underline{R})=1$ leads to a contradiction. 
The only possible way to move a pebble with the use of $R$'s pebbles to the neighbourhood of $\overline{L}$ is to consume a pebble from $\underline{L}$. On the other hand,  $\underline{L}$ is a neighbour of $\overline{L}$, so we are not able to increase the number of pebbles at $\overline{L}$'s neighbourhood. So $\overline{L}$ is not reachable under $p$ in the $n$-prism, which contradicts that $p$ is an optimal distribution of that graph. 
\end{proof}

Finally we prove the exact values. We use the previous two lower bounds. In the second and third cases we make some tricky constructions to show that the sufficient condition of the stronger bound holds.

\begin{theorem} The optimal rubbling numbers for the $n$-prisms are:
$$\varrho_{\opt}(C_{3k-1} \square P_2)=\varrho_{\opt}(P_{3k-2}\square P_2)=2k,$$
$$\varrho_{\opt}(C_{3k} \square P_2)=\varrho_{\opt}(P_{3k-1} \square P_2)=2k,$$
$$\varrho_{\opt}(C_{3k+1} \square P_2)=\varrho_{\opt}(P_{3k} \square P_2)=2k+1,$$

\begin{center}
except $\varrho_{\opt}(C_{3} \square P_2)=3.$
\end{center}

\label{kortetel}
\end{theorem}

\begin{proof}
Let $p$ be an optimal pebbling distribution of $C_n\square P_2$, and let $C$ be a rung which is not 2-reachable under $p$. Denote the neighbouring rungs of $C$ with $L$ and $R$.

\begin{figure}[ht]
\centering
\scalebox{0.6}{\input{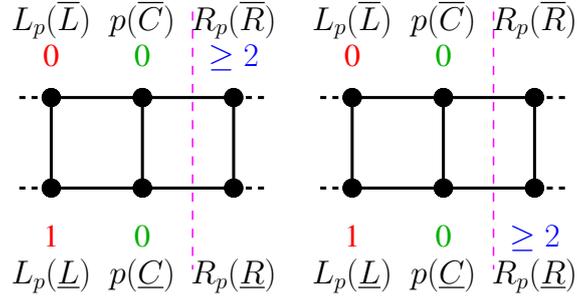}}
\caption{Each number is the value of the expression written next to it.}
\label{letra2}
\end{figure}

Let $n=3k+r$ where $n,k,r \in \mathbb{N}, 0\le r <3$. Now we show for each $r$ that the theorem holds.
We apply Lemma \ref{korlatok} in each case.

\par When $r=0$ then we use the general bound of Lemma $\ref{korlatok}$, which always holds: 
$$\varrho_{\opt}(C_n \square P_2)\geq \varrho_{\opt}(P_{n-2} \square P_2)=2(k-1)+2=2k.$$

\par When $r=1$ we show that in Lemma \ref{korlatok} the better bound holds, since at least one of the conditions fail to hold. Indirectly, assume that $\varrho_{\opt}(C_n \square P_2)=\varrho_{\opt}(P_{n-2}\square P_2)$ . This implies that all conditions in the Lemma hold, so we have one of the cases shown on Fig. \ref{letra2}. Delete the edges between $R$ and $C$ to obtain $P_n\square P_2$, place one more pebble at $\overline{C}$.  This modification of $p$ is clearly solvable on $P_n\square P_2$. This implies:
\begin{align*}
2k+2&=\varrho_{\opt}(P_n\square P_2)\leq \varrho_{\opt}(C_n \square P_2)+1 =\\
&=\varrho_{\opt}(P_{n-2}\square P_2)+1=2(k-1)+2+1=2k+1,
\end{align*}
which is a contradiction. Hence: $$\varrho_{\opt}(C_n \square P_2)\geq \varrho_{\opt}(P_{n-1} \square P_2)=2k+1.$$ 

\begin{figure}[htb]
\centering
\scalebox{0.6}{\input{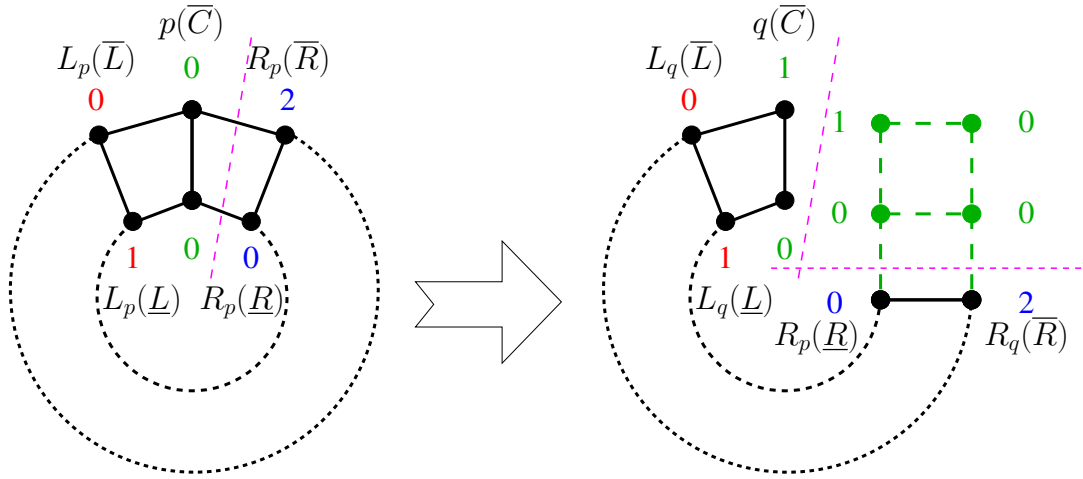}}
\caption{The numbers near rung $C$ and near the augmented part, which has dashed edges,  show how many pebbles we place at the corresponding vertex to get a solvable distribution.}
\label{kiegeszit}
\end{figure}

When $r=2$ assume for a contradiction that $\varrho_{\opt}(C_n \square P_2)=\varrho_{\opt}(P_{n-2}\square P_2)$, then $p$ fulfills the all properties. Delete the edges between $C$ and $R$ again, but now  augment the graph as shown see Fig. \ref{kiegeszit}. We place an extra pebble at $\overline{C}$ and another at the end of the new part. The augmented $P_{n+2}\square P_2$ graph is solvable under this distribution. Using Theorem \ref{letratetel} these imply:
\begin{align*}
2k+4=2(k+1)+2=\varrho_{\opt}(P_{n+2}\square P_2)&\leq \varrho_{\opt}(C_n \square P_2)+2= \\
&=\varrho_{\opt}(P_{n-2}\square P_2)+2=2k+3
\end{align*}
which is a contradiction. Hence 
$$\varrho_{\opt}(C_n \square P_2)\geq \varrho_{\opt}(P_{n-1} \square P_2)=2k.$$
\end{proof}

\section{Optimal rubbling number of the M\"obius-ladder}

\begin{theorem}
\label{Mobiustetel}
The M\"obius-ladder with length $n$ and the $n$-prism has the same optimal rubbling number. 
\end{theorem}

\begin{proof}
Similarly like in the pebbling case \cite{ladder}, the proof for the
M\"obius-ladder works almost the same as in the $n$-prism case.

For the optimal distributions of small graphs see Fig.~\ref{M_kicsi}, it is not too difficult to show with a short case analysis that these are optimal.
The proof given for the lower bound of Theorem \ref{kortetel} works
for this theorem without any change.  The upper bound also comes from
the optimal pebbling distributions of $P_n\square P_2$ like in the prism case.
\end{proof}

\begin{figure}[htb]
\includegraphics[scale=0.6]{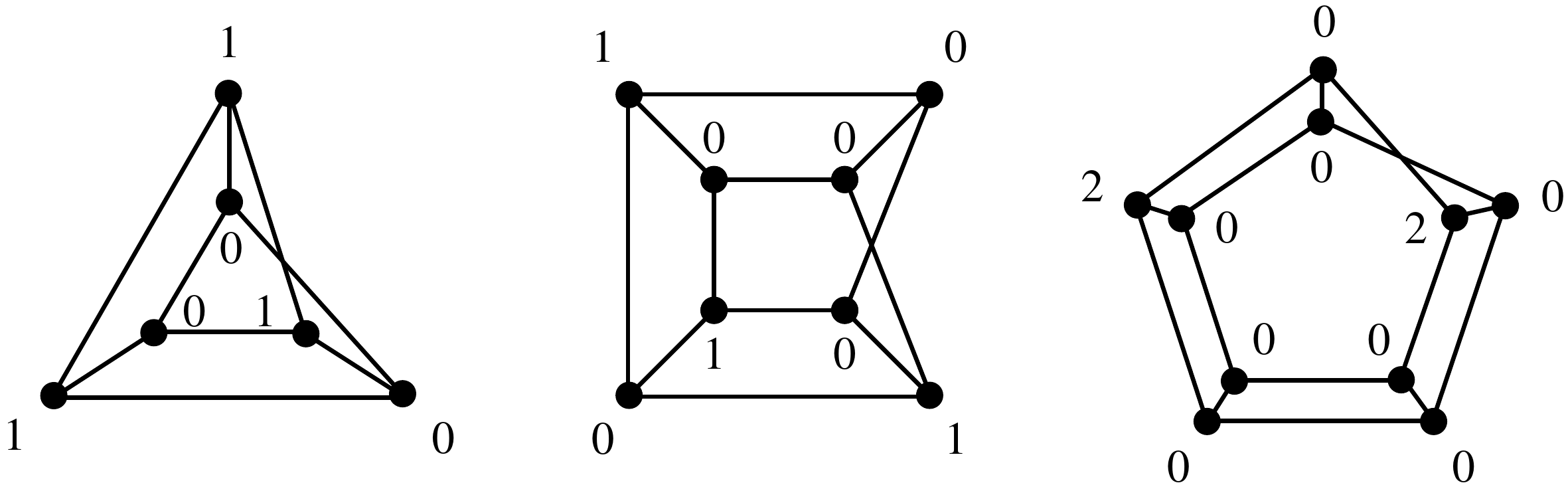}
\centering
\caption{Optimal pebbling distributions of the M\"obius-ladder when $n<6$.}
\label{M_kicsi}
\end{figure}


\newpage

\section*{Appendix: Reduction methods for the $p(R)\geq 4$ cases}
\setcounter{section}{0}

These reduction methods are used when the $R$ subgraph has at least 4 pebbles. We have shown a method already in the article for the case $p(\overline{l})+p(\overline{x})\geq 4$. Therefore we assume that $p(R)\geq 4$ and $p(\overline{l})+p(\overline{x})\leq 3$. We also suppose that the following inequalities hold:
\begin{itemize}
\item $p(\overline{l})\geq p(\overline{r})$
\item $p(\overline{l}) + p(\overline{x}) \geq p(\underline{l}) + p(\underline{x})$
 \item $p(\overline{l}) + p(\overline{x}) \geq p(\underline{r}) + p(\underline{x})$
\item $\mbox{max}(p(\overline{l}),p(\overline{x}))\geq\mbox{max}(p(\underline{l}),p(\underline{x}))  \mbox{ if }p(\overline{l}) + p(\overline{x}) = p(\underline{l}) + p(\underline{x})$
\end{itemize}
These inequalities simply state that none of the images of the $\{\overline{l},\overline{x}\}$ subgraph in $R$ contains more pebbles when we apply a nontrivial symmetry transformation.

\par During this appendix we show calculations that prove reduction methods. For simplicity  instead of $p(v)$ (the number of pebbles at vertex $v$) we only write $v$.

\subsection*{Reduction method I.}
Assume the following constraints:
\begin{itemize}
\item
$\overline{x}+\overline{l}\geq 2$ 
\item $\underline{x}+\underline{r}\geq 2$
\end{itemize}

\begin{figure}[hb]
\centering
\scalebox{0.8}{\input{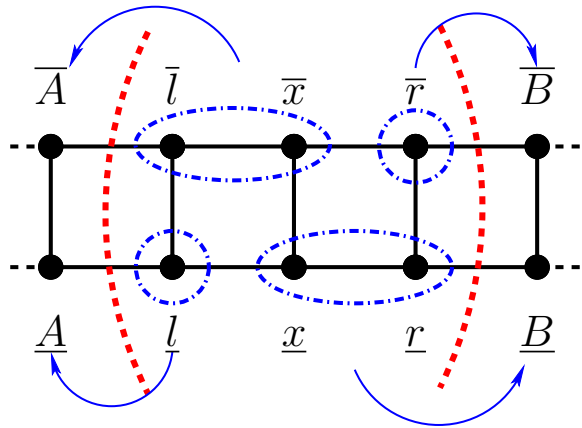}}
\caption{Reduction method I.}
\label{method_1}
\end{figure}

\par The method: 
\par Move the pebbles as shown on Fig.~\ref{method_1} then remove a pebble from $\overline{A}$ and an other one from $\underline{B}$. The fallowing calculations show that this is a good reduction:
$$R_{p^r}(\overline{A})-R_p(\overline{A})\geq \frac{1}{2}\overline{l}+\frac{3}{4}\overline{x}+\frac{3}{8}\overline{r}+\frac{1}{4}\underline{l}+\frac{1}{8}\underline{x}+\frac{3}{16}\underline{r}-1-\frac{1}{4}\geq$$
$$\geq \frac{1}{2}\underbrace{(\overline{l}+\overline{x}-2)}_{\geq 0}+\frac{1}{8}\underbrace{(\underline{x}+\underline{r}-2)}_{\geq 0}+\frac{1}{4}\underline{l}+\frac{3}{8}\overline{r}\geq 0$$  

$$R_{p^r}(\underline{A})-R_p(\underline{A})\geq \frac{1}{4}\overline{l}+\frac{3}{8}\overline{x}+\frac{3}{16}\overline{r}+\frac{1}{2}\underline{l}+\frac{1}{4}\underline{x}+\frac{3}{8}\underline{r}-\frac{1}{2}-\frac{1}{2}\geq$$
$$\geq \frac{1}{4}\underbrace{(\overline{l}+\overline{x}-2)}_{\geq 0}+\frac{1}{4}\underbrace{(\underline{x}+\underline{r}-2)}_{\geq 0}+\frac{1}{2}\underline{l}+\frac{3}{16}\overline{r}\geq 0$$ 

It is easy to see that the calculations for vertices $\overline{A}$ and $\underline{B}$ are the same, because the method is symmetric. The same holds for $\underline{A}$ and $\overline{B}$. 

\subsection*{Reduction method II.}
Assume that not all of the constrains of Method I. hold, but the following constraints hold:
\begin{itemize}
\item $p(R)\geq 5$
\item $\overline{x}+\overline{l}=3$
\item $\underline{x}+\underline{r}\geq 1$
\end{itemize}

\begin{figure}
\centering
\scalebox{0.8}{\input{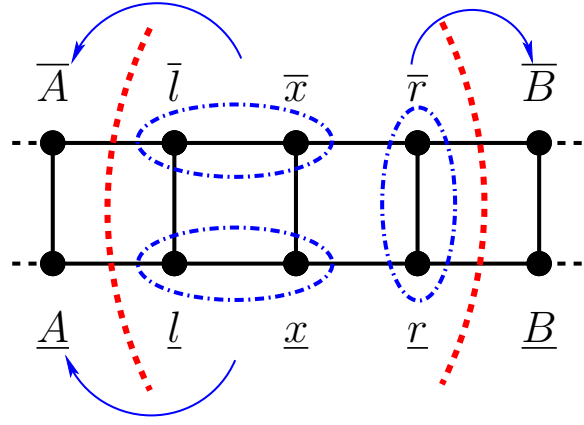}}
\caption{Reduction method II.}
\label{method_2}
\end{figure}

\par The method:
\par Move the pebbles as shown on Fig.~\ref{method_2} then remove two pebbles from $\overline{A}$.
$$R_{p^r}(\overline{A})-R_p(\overline{A})\geq \frac{1}{2}\overline{l}+\frac{3}{4}\overline{x}+\frac{3}{8}\overline{r}+\frac{1}{4}\underline{l}+\frac{3}{8}\underline{x}+\frac{7}{16}\underline{r}-2\geq$$
$$\geq \frac{1}{2}\underbrace{(\overline{l}+\overline{x}-3)}_{\geq 0}+\frac{1}{4}\underbrace{(\overline{r}+\underline{l}+\underline{x}+\underline{r}-2)}_{\geq 0}\geq 0$$  

$$R_{p^r}(\underline{A})-R_p(\underline{A})\geq \frac{1}{4}\overline{l}+\frac{3}{8}\overline{x}+\frac{3}{16}\overline{r}+\frac{1}{2}\underline{l}+\frac{3}{4}\underline{x}+\frac{1}{8}\underline{r}-1\geq$$
$$\geq \frac{1}{4}\underbrace{(\overline{l}+\overline{x}-3)}_{\geq 0}+\frac{1}{8}\underbrace{(\overline{r}+\underline{l}+\underline{x}+\underline{r}-2)}_{\geq 0}\geq 0$$ 

$$R_{p^r}(\overline{B})-R_p(\overline{B})\geq \frac{3}{8}\overline{l}+\frac{1}{4}\overline{x}+\frac{1}{2}\overline{r}+\frac{3}{16}\underline{l}+\frac{1}{8}\underline{x}+\frac{3}{4}\underline{r}-1\geq$$
$$\geq \frac{1}{4}\underbrace{(\overline{l}+\overline{x}-3)}_{\geq 0}+\frac{1}{8}\underbrace{(\overline{r}+\underline{l}+\underline{x}+\underline{r}-2)}_{\geq 0}\geq 0$$  

$$R_{p^r}(\underline{B})-R_p(\underline{B})\geq \frac{3}{16}\overline{l}+\frac{1}{8}\overline{x}+\frac{1}{4}\overline{r}+\frac{3}{8}\underline{l}+\frac{1}{4}\underline{x}-\frac{1}{2}\geq$$
$$\geq \frac{1}{8}\underbrace{(\overline{l}+\overline{x}-3)}_{\geq 0}+\frac{1}{4}\underbrace{(\overline{r}+\underline{l}+\underline{x}-1)}_{\geq 0}\geq 0$$  
\par
To show the last inequality: Since the contstrains of Method I. cannot hold now, we have $\underline{x}+\underline{r}\le 1$ which gives $\underline{r}\leq 1$. Now the pigeonhole principle implies that at least one of the vertices $\overline{r}$, $\underline{l}$, $\underline{x}$ contains a pebble under $p$.

\subsection*{Reduction method III.}
Assume that not all of the constrains of Method I. hold, not all of the constrains of Method II. hold, but the following constraints hold:
\begin{itemize}
\item $\overline{l}+\overline{x}=2$
\item $\underline{r}+\underline{x}\leq 1$
\item $\underline{l}+\underline{x}\geq 2$
\end{itemize}

\par
The method:
\par Move the pebbles as in method II (shown on Fig.~\ref{method_2}), but now we remove a pebble from both vertices of rung $A$.

$$R_{p^r}(\overline{A})-R_p(\overline{A})\geq \frac{1}{2}\overline{l}+\frac{3}{4}\overline{x}+\frac{3}{8}\overline{r}+\frac{1}{4}\underline{l}+\frac{3}{8}\underline{x}+\frac{7}{16}\underline{r}-1-\frac{1}{2}\geq$$
$$\geq \frac{1}{2}\underbrace{(\overline{l}+\overline{x}-2)}_{\geq 0}+\frac{1}{4}\underbrace{(\underline{l}+\underline{x}-2)}_{\geq 0}\geq 0$$  

$$R_{p^r}(\underline{A})-R_p(\underline{A})\geq \frac{1}{4}\overline{l}+\frac{3}{8}\overline{x}+\frac{3}{16}\overline{r}+\frac{1}{2}\underline{l}+\frac{3}{4}\underline{x}+\frac{1}{8}\underline{r}-1-\frac{1}{2}\geq$$
$$\geq \frac{1}{4}\underbrace{(\overline{l}+\overline{x}-2)}_{\geq 0}+\frac{1}{2}\underbrace{(\underline{l}+\underline{x}-2)}_{\geq 0}\geq 0$$ 

$$R_{p^r}(\overline{B})-R_p(\overline{B})\geq \frac{3}{8}\overline{l}+\frac{1}{4}\overline{x}+\frac{1}{2}\overline{r}+\frac{3}{16}\underline{l}+\frac{1}{8}\underline{x}+\frac{3}{4}\underline{r}-\frac{1}{2}-\frac{1}{4}\geq$$
$$\geq \frac{1}{4}\underbrace{(\overline{l}+\overline{x}-2)}_{\geq 0}+\frac{1}{8}\underbrace{(\underline{l}+\underline{x}-2)}_{\geq 0}\geq 0$$  

$$R_{p^r}(\underline{B})-R_p(\underline{B})\geq \frac{3}{16}\overline{l}+\frac{1}{8}\overline{x}+\frac{1}{4}\overline{r}+\frac{3}{8}\underline{l}+\frac{1}{4}\underline{x}-\frac{1}{2}-\frac{1}{4}\geq$$
$$\geq \frac{1}{8}\underbrace{(\overline{l}+\overline{x}-2)}_{\geq 0}+\frac{1}{4}\underbrace{(\underline{l}+\underline{x}-2)}_{\geq 0}\geq 0$$  

\subsection*{Reduction method IV.}
Assume that not all of the constrains of Method I. hold, not all of the constrains of Method II. hold, not all of the constrains of Method III. hold, but the following constraints hold:
\begin{itemize}
\item $\overline{l}+\overline{x}=2$
\item $\underline{r}+\underline{x}\leq 1$
\item $\overline{r}+\underline{r}\geq 2$
\item $p(R)\geq 5$
\end{itemize}

\par The method:
\par
Move the pebbles as in method II (shown on Fig.~\ref{method_2}),  remove one pebble from $\overline{A}$ and one pebble from $\overline{B}$.

$$R_{p^r}(\overline{A})-R_p(\overline{A})\geq \frac{1}{2}\overline{l}+\frac{3}{4}\overline{x}+\frac{3}{8}\overline{r}+\frac{1}{4}\underline{l}+\frac{3}{8}\underline{x}+\frac{7}{16}\underline{r}-1-\frac{1}{2}\geq$$
$$\geq \frac{1}{2}\underbrace{(\overline{l}+\overline{x}-2)}_{\geq 0}+\frac{3}{8}\underbrace{(\overline{r}+\underline{r}-2)}_{\geq 0}\geq 0$$  

$$R_{p^r}(\underline{A})-R_p(\underline{A})\geq \frac{1}{4}\overline{l}+\frac{3}{8}\overline{x}+\frac{3}{16}\overline{r}+\frac{1}{2}\underline{l}+\frac{3}{4}\underline{x}+\frac{1}{8}\underline{r}-\frac{1}{2}-\frac{1}{4}\geq$$
$$\geq \frac{1}{4}\underbrace{(\overline{l}+\overline{x}-2)}_{\geq 0}+\frac{1}{8}\underbrace{(\overline{r}+\underline{r}-2)}_{\geq 0}\geq 0$$ 

$$R_{p^r}(\overline{B})-R_p(\overline{B})\geq \frac{3}{8}\overline{l}+\frac{1}{4}\overline{x}+\frac{1}{2}\overline{r}+\frac{3}{16}\underline{l}+\frac{1}{8}\underline{x}+\frac{3}{4}\underline{r}-1-\frac{1}{2}\geq$$
$$\geq \frac{1}{4}\underbrace{(\overline{l}+\overline{x}-2)}_{\geq 0}+\frac{1}{2}\underbrace{(\overline{r}+\underline{r}-2)}_{\geq 0}\geq 0$$  

$$R_{p^r}(\underline{B})-R_p(\underline{B})\geq \frac{3}{16}\overline{l}+\frac{1}{8}\overline{x}+\frac{1}{4}\overline{r}+\frac{3}{8}\underline{l}+\frac{1}{4}\underline{x}-\frac{1}{2}-\frac{1}{4}\geq$$
$$\geq \frac{1}{8}\underbrace{(\overline{l}+\overline{x}-3)}_{\geq 0}+\frac{1}{4}\underbrace{(\underline{r}-1)}_{\geq 0}\geq 0$$  

The last inequality holds for the same reason as in the proof of method II.

\subsection*{Analysis of the case set where one of the methods works}
\begin{statement2}
The descirbed methods handle all cases when $p(R)\geq 5$. 
\end{statement2}
The inequalities of section 1 result that $3\geq \overline{l}+\overline{x}\geq 2$. Method I and II cover the cases when $\overline{l}+\overline{x}=3$. Furthermore Method I, III and IV cover the cases when $\overline{l}+\overline{x}=2$. 
\par The methods also work for many of cases when $p(R)=4$, but unfortunately not for all. 
\begin{figure}[ht]
\centering
\scalebox{0.45}{
\input{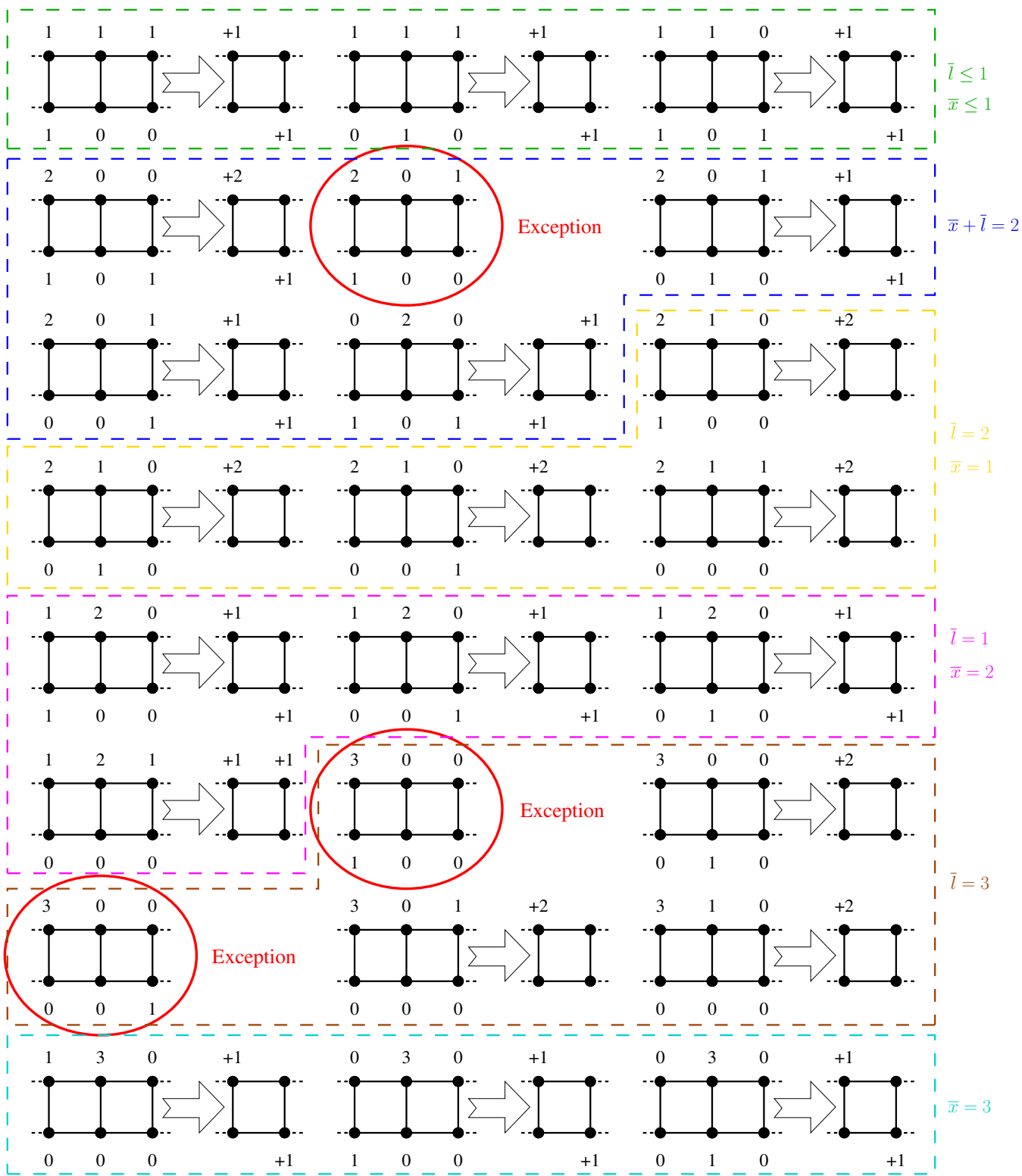}
}
\label{remaining}
\caption{Reduction methods for the remaining cases. The three case in the red circles are exceptions. The $P_3\square K_2$ graphs are the $R$ subgraphs before the reduction. The square graphs are the subgraphs conatining rungs $A$ and $B$.}

\end{figure}
\subsection*{Reduction methods for the remaining cases}

The remaining cases with their reduction method are shown on Fig.~\ref{remaining}. Three of these cases are exceptions, which doesn't have a reduction method that fulfills the modified inequalities for all the inspected four vertices. The figure shows the logic of the enumeration, hence it is easy to check that none of the possible cases are missing.


\section*{Acknowledgment}
Research is supported by the Hungarian National Research Fund (grant
number K108947). We thank the anonymous referee for the valuable
suggestions to improve the presentation of the paper.  
\end{document}